\documentclass{amsart}
\usepackage[margin=1.25in,dvips]{geometry}

\usepackage{cite}
\usepackage{amssymb,mathrsfs,graphicx}
\usepackage{amsmath}
\usepackage{xcolor}
\usepackage{soul}

\numberwithin{equation}{section}

\title[Decomposition of global solutions for a class of nonlinear wave equations ]{Decomposition of global solutions for a class of nonlinear wave equations}

 \author[Georgios Mavrogiannis]{Georgios Mavrogiannis}
 \address[Georgios Mavrogiannis]{\newline
         Department of Mathematics, \newline
         Rutgers University, New Brunswick, NJ 08903 USA.}
  \email[]{gm758@math.rutgers.edu}

 \author[Avy Soffer]{Avy Soffer}
 \address[Avy Soffer]{\newline
        Department of Mathematics, \newline
         Rutgers University, New Brunswick, NJ 08903 USA.}
  \email[]{soffer@math.rutgers.edu}

\author[Xiaoxu Wu]{Xiaoxu Wu}
\address[Xiaoxu Wu]{\newline
        Mathematical Sicences Institute, \newline
        Australia National University, Acton, ACT 2601, Australia}
 \email[]{Xiaoxu.Wu@anu.edu.au}

\newtheorem{theorem}{Theorem}[section]
\newtheorem{lemma}{Lemma}[section]

\newtheorem{assumption}{Assumption}[section]

\newtheorem{proposition}{Proposition}[section]
\newtheorem{remark}{Remark}[section]

\newtheorem{definition}{Definition}[section]
\newcommand{\p}{\partial}

\newcommand{\V}{\mathcal{V}}

\newcommand{\N}{\mathcal{N}}
\newcommand{\R}{\mathbb{R}}
\newcommand{\Hi}{\mathcal{H}}

\usepackage{xcolor}

\newcommand{\vu}{\vec{u}}

\newcommand{\eq}{\begin{equation}}
\newcommand{\eeq}{\end{equation}}

\begin{document}

\date{\today}



\begin{abstract}
In the present paper we consider global solutions of a class of non-linear wave equations of the form
\begin{equation*}
    \Box u=  N(x,t,u)u,
\end{equation*}
where the nonlinearity~$ N(x,t,u)u$ is assumed to satisfy appropriate boundedness assumptions.

Under these appropriate assumptions we prove that the free channel wave operator exists. Moreover, if the interaction term~$N(x,t,u)u$ is localised, then we prove that the global solution of the full nonlinear equation can be decomposed into a `free' part and a `localised' part.

\end{abstract}
\maketitle
\centerline{\date}

\tableofcontents

\section{Introduction}\label{sec: intro}

In this paper we prove two decomposition results, see already Theorems~\ref{main theorem 1},~\ref{main theorem 3}, for global solutions of the wave equation 
\begin{equation}\label{eq: wave}
\Box u(x,t) = N(x,t,u)u(x,t), \quad (x,t) \in \mathbb{R}^{n}\times\mathbb{R}_{\geq 0}, \quad n\geq 3,
\end{equation}
with~$\Box = -\partial_t^2 +\Delta_x$. We prove our decomposition results under certain assumptions on the global behavior of the solution~$u$ and of the nonlinearity~${N}(x,t,u)u$.

Specifically, we note that 
\begin{itemize}
    \item In the case of Theorem~\ref{main theorem 1} we assume that the solution~$u$ satisfies the Assumptions~\ref{asp: global},~\ref{asp: N}.\\
    \item In the case of Theorem~\ref{main theorem 3} we assume that the solution~$u$ satisfies the Assumptions~\ref{asp: global},~\ref{asp: N: non-local}. 
\end{itemize}
For the exact Assumptions see already Section~\ref{subsec: sec: intro, subsec 1}.

\subsection{Background and the aim of this paper}

We are considering Nonlinear wave equations (WE) with general interaction terms of the form~\eqref{eq: wave}. We allow such interactions to be both linear and nonlinear, and may explicitly depend on space and time.

The interest and importance of wave equations is well known, and have many applications in Physics and Geometry. The classical examples are vibrations of strings and membranes, light waves and acoustic wave propagation as well as the dynamics of space-time itself, as in General Relativity. We are particularly interested in looking for the properties of solutions of the wave equation~\eqref{eq: wave}, with general interactions. This analysis is relevant to the analysis of \emph{systems of equations}.

The aim of this work is to study the large time behavior of global solutions of wave equations with general interactions, under several general assumptions on the nonlinearity and the global behavior of the solution. In particular if we assume that the interaction is localized around the origin, we will prove that global solutions of equations of the form~\eqref{eq: wave} break, as time goes to infinity, to a part that evolves like a solution of the free wave equation, and a (weakly) localized part around the support of the interaction.

 This is known as the key step in Scattering theory, leading to the property of Asymptotic Completeness (AC) which, in laypersons terms, is the generic statement that all asymptotic states of systems are composed of a set of subsystems with known (simpler) dynamics, who are moving independently as time goes to infinity.

Under our assumptions, we show that the free wave part separates from the rest of the solution for a large class of nonlinearities. When the number of such subsystems is more than one then this corresponds to a multi-channel scattering problem. To rigorously formulate multichannel scattering problems one needs to use the notion of Channel wave operators, see e.g.\cite{enss1978asymptotic,reed1979iii,sigal1987n,  sigal1990long, SW20221, SW-Klein,SW-3body} and cited references. The multichannel scattering problem is intrinsic when the interaction terms can create localized solutions, and allow free waves at the same time. For this, the interaction terms should have attractive forces~(focusing). Note that in the present paper we construct the free channel wave operator, which can be seen as the first step of studying a multi-channel scattering problem. See also~\cite{SW-LD} for a proof of dispersive estimates for Schr\"odinger equations with time-quasi-periodic potentials, utilizing the concept of channel wave operators.

In the case when the solution scatters to a free wave only (one channel case), a lot of results are known, going back to early results of Morawetz~\cite{M1968} and Strauss~\cite{S1968} in the early 1960's. The development of Strichartz estimates~\cite{Strichartz} then enables to get a fairly complete and general understanding of one channel problems. That includes small data cases as well as defocusing non-linear terms. See also~\cite{GGH} for the application of Fredholm analysis to time-dependent Schr\"odinger equations.

\subsection{Discussion on our results}

In both our Theorems~\ref{main theorem 1},~\ref{main theorem 3}, we prove the existence of the free channel wave operator. In particular, we do not need radial symmetry assumptions in order to construct the free channel. The assumption of spherical symmetry or localization of the interaction terms is used only to control the localization properties of the non-free part. Clearly, with this kind of generality, one cannot expect to have the standard a-priori estimates. The approach is then to prove directly the existence of a free wave at infinity that separates from the rest of the solution. This is done by proving that, in a suitably defined Hilbert space, as time goes to infinity the solution converges in the strong topology to a free wave plus an asymptotically orthogonal part.
To achieve this, we project the full solution on a subspace of the phase-space (we microlocalize) where only the free solution can concentrate.

In our Theorem~\ref{main theorem 1} we in fact prove more, by assuming an appropriate localization property for the nonlinearity. In particular, we prove that the solution can be decomposed to a free part and a localized part. 

Throughout this paper, $C$ will denote a constant and may vary from one line to another. We write $\lesssim$ or $\gtrsim$ whenever $A\leq CB$ or $CA\geq B$ for some constant $C>0$. We write $A\lesssim_a B$ or $A\gtrsim_a B$ if  $A\leq C_aB$ or $C_aA\geq B$ for some constant $C_a>0$ which depends on parameter $a$.

To elaborate further, we prove that the localized part is localized in a weighted integrated sense:
$$
\|\langle x\rangle^{\delta}\nabla u_{lc}\|_{L^2}+\|\langle x\rangle^{\delta}\partial_t u_{lc}\|_{L^2} \leq C,
$$
for some positive $\delta$ and $C$. Specifically, see already inequality~\eqref{wlc: chiN} of Theorem~\ref{main theorem 1}.

\subsection{Comparison with relevant work}

It is instructive to compare our results to the works on wave equations with general initial data in $\dot H^1$ in the energy critical case.

Kenig and Merle pioneered the study of the wave equation with energy-critical nonlinearity in~\cite{kenig2008global}. Building on this foundation, a series of subsequent works, including significant collaborations with Collot and Duyckaerts~\cite{duyckaerts2014scattering, collot2024soliton, duyckaerts2019soliton} for dimensions 6, 3, and all odd spatial dimensions, as well as references cited in~\cite{collot2024soliton}, have advanced the field. Complementing these efforts, the insightful work by Jendrej and Lawrie~\cite{JL2023} in dimensions 4 and higher has played a key role in fully resolving the asymptotic dynamics of the wave equation with an energy-critical interaction term. Together, these contributions cover all dimensions 3 and above for spherically symmetric initial data in $\dot{H}^1$.

It is shown that any solution of the equation that does not blowup in a finite time, converges (in $\dot H^1$) to a free wave plus possibly the (unique) ground state soliton of the equations and plus self similar solutions, all of which with envelope given by the ground state soliton. The blowup solutions are also self-similar, with the same soliton serving as their envelope. Here, we adopt a generalized notion of self-similarity: a function is considered self-similar if it can be expressed as a function of $x/t^{\alpha}$, where $\alpha \neq 0$, or as a function of $x/\lambda(t)$, where $\lambda(t)$ is a real-valued function. A similar progress was achieved in the study of the wave map dynamics, see e.g. \cite{jendrej2020continuous,jendrej2021soliton}
and also in the case of potential perturbations of the energy critical wave equation \cite{jia2017generic}.

In contrast, we consider the wave equation~\eqref{eq: wave} with initial data in $\dot H^1$ but with a general type of interaction term, which satisfies several assumptions~(that differ from the assumptions of the papers mentioned above). Our class of allowed interactions is defined in terms of abstract properties, and allows for space and time dependence of the interaction. Yet, our assumptions preclude some cases, in particular the energy critical case.
In Theorem~\ref{main theorem 1} we prove that any global solution of equation~\eqref{eq: wave} converges in $\dot H^1$ to a sum of a free wave and a \underline{localized in space} solution, under our assumptions. However, we do not prove that these localized solutions are time independent. Any such solution, if it is time independent, has to be a soliton. In contrast, in the works mentioned above, on the critical case, a subtle theorem is proved that excludes the possibility of time dependent localized solutions (so-called breathers) without the presence of a free wave.
Proving/disproving the existence of breathers in higher dimensions for general type of equations is a major, fundamental problem which is still open.

Furthermore, our notion of blow-up is qualitative; it only means that the $\dot H^1$ semi-norm of the solution grows to infinity with time (at a sufficiently fast rate). If the localized part is time independent (as is conjectured for NLWEs in higher dimensions), then this localized part is a soliton~(a solution of the time independent equation).

\subsection{Some useful spaces and operators}

We denote 
\eq
\vec{u}=(u,\dot{u}),\qquad \dot{u} = \partial_t u,\qquad
\mathcal{N}(x,t,\vec{u}):=\begin{pmatrix}
    0 & 0 \\
    N(x,t,u) & 0 
\end{pmatrix}.
\eeq

We define the Banach space 
\begin{equation}
    (\Hi,\|\cdot\|_{\mathcal{H}})
\end{equation}
where
\begin{equation}
\Hi:=\dot{H}_x^1(\mathbb{R}^n)\times L^2_x(\mathbb{R}^n),\qquad \|\vu\|_\Hi=\sqrt{\| pu\|^2+\| \dot{u}\|^2} 
\end{equation}
and 
\begin{equation}
   p:=-i\nabla_x.
\end{equation}

Furthermore, we define the inner product
\eq
(\vu, \vec{v})_{\Hi}:=(pu_1, pv_1)_{L^2_x(\mathbb{R}^n)}+(u_2, v_2)_{L^2_x(\mathbb{R}^n)}\quad \text{ for }\vu, \vec{v}\in \Hi.
\eeq
In what follows, we use the notation 
\begin{equation}
    \|\cdot\|
\end{equation}
to denote either
\begin{equation}
    \|\cdot\|_{L^2_x(\mathbb{R}^n)}\quad\text{or}\quad\|\cdot\|_{L^2_x(\mathbb{R}^n)\to L^2_x(\mathbb R^n)}
\end{equation}
depending on the context.

Let $\langle x\rangle:=\sqrt{1+|x|^2}$. For~$\sigma\in \mathbb{R}$ we define 
\begin{equation}
    \begin{aligned}
\dot H^1_\sigma(\mathbb{R}^n)&:=\{ f: \langle x\rangle^{\sigma}
|p|f\in L^2_x(\mathbb{R}^n)\},\\
L^2_\sigma(\mathbb{R}^n)&:=\{ f: \langle x\rangle^\sigma f\in L^2_x(\mathbb{R}^n)\}.
    \end{aligned}
\end{equation}

Furthermore, we define the Banach space 
\begin{equation}
    \Hi_{\sigma}:=\dot H_{\sigma}^1(\mathbb{R}^n)\times L^2_{\sigma
}(\mathbb{R}^n).
\end{equation}

We define the operator~$\mathcal{X}: \Hi_{1}\to \Hi$ with formula
\eq
\mathcal{X}:=\begin{pmatrix}
    |p|^{-1}\langle x\rangle |p| & 0 \\
    0& \langle x\rangle
\end{pmatrix},\label{def: chi}
\eeq
and also use the notation
\begin{equation}
    \mathcal X^\sigma:=\begin{pmatrix}
    |p|^{-1}\langle x\rangle^\sigma |p| & 0 \\
    0& \langle x\rangle^\sigma
\end{pmatrix},\qquad \sigma>0.
\end{equation}

\subsection{The assumptions of the main Theorems}\label{subsec: sec: intro, subsec 1}

As discussed earlier, we prove our main Theorems by imposing the following assumptions:
\begin{itemize}
    \item For Theorem~\ref{main theorem 1} we need Assumptions~\ref{asp: global} and~\ref{asp: N}.\\
    \item For Theorem~\ref{main theorem 3} we need Assumptions~\ref{asp: global} and~\ref{asp: N: non-local}.
\end{itemize}

\begin{assumption}[Global well-posedness in $\Hi$]\label{asp: global} The initial condition 
\begin{equation}
    \vu(x,0):=(u(x),\dot{u}(x))\in \Hi
\end{equation}
gives rise to a global solution in $\Hi$: 
\begin{equation}\label{eq: asp: global, eq 1}
    E:=\sup\limits_{t\geq0} \| \vu(x,t)\|_{\Hi
}<\infty
\end{equation}
and $u(t,x)\to 0$ as $|x|\to \infty$ for all $t\geq0$. This decay condition is specifically required for the application of the Gagliardo-Nirenberg inequality in the proofs of the main theorems.

\end{assumption}
\begin{assumption}[Assumption on $\mathcal{N}(x,t,\vu)\vu$: local interactions]\label{asp: N} We assume that for some $\sigma>1$, the following holds:
\eq
\sup\limits_{t\geq 0}\| \mathcal X^\sigma \mathcal{N}(x,t,\vu(t))\vu(t) \|_{\mathcal H}<\infty,
\eeq
or equivalently
\eq
\sup\limits_{t\geq 0}\| \langle x\rangle^\sigma N(x,t,u(t))u(t) \|<\infty.
\eeq
\end{assumption}

\begin{assumption}[Assumption on $\mathcal N(x,t,\vu)\vu$: non-local interactions]\label{asp: N: non-local} We assume that 
\eq
\sup\limits_{t\geq 0}\| N(x,t,u(t))u(t)\|_{L^1_x(\mathbb{R}^n)}<\infty.
\eeq

\end{assumption}

\subsection{Examples} We now give typical examples of nonlinearities that satisfy Assumptions~\ref{asp: N},~\ref{asp: N: non-local}. First, we state a classical Gagliardo-Nirenberg inequality

\begin{theorem}[classical Gagliardo-Nirenberg] Let $1\leq q\leq \infty$ be a positive extended real quantity. Let $j$ and $m$ be non-negative integers such that $j<m$. Furthermore, let $1\leq r\leq\infty$ be a positive extended real quantity, $p\geq 1$ be real and $\theta\in [0,1]$ such that the relations
\begin{equation}
    \frac{1}{p}=\frac{j}{n}+\theta \left( \frac{1}{r}-\frac{m}{n}\right)+\frac{1-\theta}{q},\qquad \frac{j}{m}\leq \theta\leq 1
\end{equation}
hold. Then,
\begin{equation}
    \|D^j u\|_{L^p(\mathbb{R}^n)}\leq C\|D^m u\|^\theta_{L^r(\mathbb{R}^n)}\| u\|_{L^q(\mathbb{R}^n)}^{1-\theta}
\end{equation}
for any $u\in L^q(\mathbb R^n)$ such that $D^mu \in L^r(\mathbb R^n)$, with two exceptional cases:

\begin{enumerate}
    \item if $j=0$, $q=+\infty$ and $rm<n$, then an additional assumption is needed: either $u(x)\to 0$ as $|x|\to \infty$, or $u\in L^s(\mathbb{R}^n)$ for some finite $s$; 
    \item if $r>1$ and $m-j-\frac{n}{r}$ is a non-negative integer, then the additional assumption $\frac{j}{m}\leq \theta<1$ is needed.
\end{enumerate}    
\end{theorem}

Now, let~$u$ be a solution of~\eqref{eq: wave} satisfying Assumption~\ref{asp: global}. By Gagliardo-Nirenberg with $j=0$, $\theta=1$, $m=1$ and $r=2$ we obtain
\begin{equation}
   \sup\limits_{t\geq 0} \| u(t)\|_{L^{2n/(n-2)}_x(\mathbb{R}^n)}\lesssim \sup\limits_{t\geq 0} \| u(t)\|_{\dot H^1_x(\mathbb{R}^n)}<\infty,\qquad n\geq 3.
\end{equation}

Furthermore, if the solution~$u$ is in addition radial, then we use a radial Sobolev embedding Theorems (see Lemma~1 of \cite{SWW2007}) and we obtain{
\eq
|u(x)|\lesssim_E \frac{1}{|x|^{\frac{n-2}{2}}},\qquad \forall\, |x|\geq 1,\qquad n\geq 3.\label{rSe: est}
\eeq}
Here, $E$ is defined in Equation~\eqref{eq: asp: global, eq 1}.

\begin{lemma}
    Let~$n=3$. Let~$u$ satisfy Assumption~\ref{asp: global} and moreover assume that the initial data are radial. Then, the nonlinearity
    \begin{equation}
        N(x,t,u)u=  b(|u|)\frac{|u|^r}{(1+|u|^2)^{q/2}}u,\qquad 3<r\leq q-1,
\end{equation}
satisfies the Assumption~\ref{asp: N}. Here~$b(\cdot)$ stands for a bounded real-valued functional.
\end{lemma}

\begin{proof} 
Let~$r\leq q-1$. We use estimate~\eqref{rSe: est} to calculate
    \begin{align}
       \| \chi(|x|\leq 1)\langle x\rangle^\sigma b(u) \frac{u^r}{(1+u^2)^{q/2}}u \|\lesssim_\sigma \sup\limits_{x\in \mathbb{R}^3}  \frac{|u|^{r+1}}{(1+u^2)^{q/2}}
       \lesssim_\sigma& 1,
    \end{align}
    for any~$\sigma >0$. 
     
    Moreover, we note that for~$r>3$ and for $\sigma \in (1, \frac{r+1}{2}-3/2)$ we obtain
    \begin{align}
         \| \chi(|x|> 1)\langle x\rangle^\sigma b(u) \frac{u^r}{(1+u^2)^{q/2}}u \|\lesssim &  \| \chi(|x|> 1)\langle x\rangle^\sigma u^{r+1} \|\nonumber\\
         \lesssim & \| \chi(|x|> 1)\langle x\rangle^\sigma \frac{1}{|x|^{(r+1)/2}}\|\nonumber\\
         \lesssim & 1.
    \end{align}
    We conclude the proof.
\end{proof}

\begin{lemma}
    Let~$n=4$ and let~$u$ be a solution of the nonlinear wave~\eqref{eq: wave} where the Assumption~\ref{asp: global} holds. 
    
    Then, the nonlinearity
    \begin{equation}
   N(x,t,u)= b(u)u^3,
\end{equation}
satisfies Assumption~\ref{asp: N: non-local}. Here~$b$ stands for a bounded functional.
\end{lemma}

\begin{proof} By Galliardo-Nirenberg inequalities and Assumption~\ref{asp: global}, we have 
\begin{equation}
    \|b(u)u^4\|_{L^1_x(\mathbb{R}^4)}\lesssim \| u\|_{L^4_x(\mathbb{R}^4)}^4 \lesssim \sup\limits_{t\geq 0} \| u(x,t)\|_{\dot H^1_x(\mathbb{R}^4)}^4<\infty.
\end{equation}
\end{proof}

\subsection{Preliminary discussion of our Theorem~\ref{main theorem 1}}\label{sec: rough}

We impose the initial condition
\begin{equation}\label{eq: initial condition}
    \vu(x,0):=(u(x),\dot{u}(x))\in \Hi.
\end{equation} 

In our main Theorem~\ref{main theorem 1}, we prove that, under Assumptions~\ref{asp: global} and~\ref{asp: N}, the evolution of the nonlinear wave equation~\eqref{eq: wave} with initial condition~\eqref{eq: initial condition} yields any global solution $\vu(x,t)$ satisfying
\begin{equation}
    \vu(x,t) \sim \vu_{\text{free}}(x,t) + \vu_{\text{lc}}(x,t),
\end{equation}
for $n \geq 3$. The terms of the RHS of the above estimate are respectively a free solution and a localized part~(the non-free part). For the justification of the name `localised part' see already inequality~\eqref{wlc: chiN}.

\subsection{Free channel wave operators}\label{subsecL free channel}

In this section we introduce the notion of the free channel wave operator, see Definition~\ref{def: subsecL free channel, def 1}, which is key in separating the free solution~$\vu_{free}$ from the full solution $\vu(x,t)$.

We define the following cut-offs. Following~\cite{SW20221}, we denote the space cut-off by $F_c$ and the frequency cut-off by $F_1$.  
\begin{definition}
We denote as~$F_c$ and $F_1$ the smooth cut-off functions that satisfy the following conditions
 \begin{equation}
     F_j(\lambda)=\begin{cases}
         1 & \text{ when }\lambda\geq 1\\
         0 & \text{ when }\lambda <1/2
     \end{cases},\qquad j=c,1.\label{def: F}
 \end{equation}
Moreover, we define $F_j(\frac{\lambda}{a}\leq 1)\equiv 1-F_j(\frac{\lambda}{a})$ and $F_j(\frac{\lambda}{a}> 1)\equiv F_j(\frac{\lambda}{a})$ for all $j=c,1$ and $a>0$.     
\end{definition}

Now, note the following definition.

\begin{definition}\label{def: ubsecL free channel, def 1}
Let~$\sigma>1$. We define a function~$F_\alpha$ to satisfy the following 
    \begin{enumerate}
\item When $N$ is a localized interaction, see Assumption~\ref{asp: N}, we define
\eq\label{Falpha1}
     F_\alpha(x,p,t):=  F_c(\frac{|x|}{t^\alpha}\leq 1)F_1(t^{\alpha/2}|p|\geq 1)F_1(|p|\leq t^{\alpha/2})\quad  
 \eeq
 for $\alpha\in (0,1-1/\sigma)$.

 \item When $N$ is a non-local interaction and $n\geq 4$, see Assumption~\ref{asp: N: non-local}, we define
 \eq\label{Falpha2}
     F_\alpha(x,p,t):= F_c(\frac{|x|}{t^\alpha}\leq 1)F_1(|p|\leq t^{\alpha}) 
 \eeq
 for $\alpha\in (0, \frac{n-3}{2n+1})$. 
 \end{enumerate}

\end{definition}

We have the immediate lemma.

\begin{lemma}\label{lem: ubsecL free channel, lem 1}
    The function $F_\alpha$, as in Definition~\ref{def: ubsecL free channel, def 1}, satisfies:
 \begin{enumerate}
\item When $N$ is a localized interaction, see Assumption~\ref{asp: N}, then we have
\eq
\| F_\alpha(x,p,t)e^{\pm it|p|} \langle x\rangle^{-\sigma}\|\lesssim_\alpha \frac{1}{t^{\sigma}},\qquad t\geq 1.
\eeq

 \item When $N$ is a non-local interaction and $n\geq 4$, see Assumption~\ref{asp: N: non-local}, then we have
 \eq
\| F_\alpha(x,p,t)e^{\pm it|p|}\|_{L^1_x(\mathbb{R}^n)\to L^2_x(\mathbb{R}^n)} \lesssim_\alpha \frac{\log_2|t|}{|t|^{(n-1)/2-(n+1)\alpha/2}} ,\qquad t\geq 2.
 \eeq
 
 \end{enumerate}
\end{lemma}

\begin{remark}
    Note that Lemma~\ref{lem: ubsecL free channel, lem 1} is equivalent to part (b) of Lemma~\ref{lem: decay} using the relations
\begin{equation}
    e^{\pm it|p|}=\cos(\pm t|p|)+i\sin(\pm t|p|),
\end{equation}
and see already the proof of Lemma~\ref{lem: decay}.
\end{remark}

Now, we are ready to define the free channel wave operator

\begin{definition}\label{def: subsecL free channel, def 1}
  Let the solution~$\vec{u}$ of the wave equation~\eqref{eq: wave} arise from initial data~$\vec{u}(0)$. Then, we formally define the free channel wave operator~$\Omega_{\alpha}^*$ acting on initial data as follows: 
 \begin{equation}\label{eq: subsecL free channel, eq 1}
 \Omega_{\alpha}^*\vu(0) := s\text{-}\lim_{t\to \infty} \mathcal{F}_\alpha(x,p,t) U_0(0,t) \vu(t)\quad \text{ in }\Hi,
 \end{equation}
 where ~$U_0(t,0)$ denotes the free solution operator of the equation \eqref{eq: wave} and $\mathcal{F}_\alpha$ denotes a matrix operator with formula
 \eq\label{Falpha: matrix}
\mathcal{F}_\alpha(x,p,t)=\begin{pmatrix}
|p|^{-1} F_\alpha(x,p,t) |p| &0 \\
 0 &  F_\alpha(x,p,t)
 \end{pmatrix},
 \eeq
 where for~$F_\alpha$ see Definition~\ref{def: ubsecL free channel, def 1}.
\end{definition}

Of course, we do not know yet whether the free channel wave operator of the above definition exists. We prove that it exists in our main Theorems. 

We have the following remarks

\begin{remark} Let $H_0:=-\Delta_x$. With foresight we note that 
 \begin{itemize}
     \item When $\alpha\in (0,1-1/\sigma)$, the weighted $L^2$ estimate holds:
 \eq
 \| F_\alpha(x,p,t)e^{it|p|}\langle x \rangle^{-\sigma} \|\in L^1_t[1,\infty).\label{freeest: 1}
 \eeq
  \item  When $\alpha \in (0,\frac{n-3}{2n+1})$, the $L^1$ estimate holds:
  \eq
  \| F_\alpha(x,p,t) e^{it|p|}\|_{L^1_x(\mathbb{R}^n)\to L^2_x(\mathbb{R}^n)}\in L^1_t[1,\infty).
  \eeq
 \end{itemize}
  The present remark will be relevant later, when we present the dispersive estimates for the free flow, see already Lemma~\ref{lem: decay}. 
 \end{remark}

\begin{remark}
The cutoff of the frequency away from infinity ensures that the dispersive estimates for the free flow hold. This poses no obstruction, as the free wave cannot concentrate at infinite frequency, thereby precluding blowup. Consequently, this same cutoff allows us to include derivative terms in the interaction, particularly space-localized metric perturbations. The reader is also encouraged to compare this approach with~\cite{rodnianski2004longtime} and~\cite{guillarmou2013resolvent}; in the latter, a multiplier method is utilized to investigate scattering and decay for Schr\"odinger’s equation on general asymptotically flat manifolds under a non-trapping condition.
\end{remark}

\section{The main theorems}\label{sec: theorems}

We are ready to present our main results.

\begin{theorem}\label{main theorem 1}
Let~$\sigma>1$. Let $\vu(t)=(u(x,t),\dot{u}(x,t))$ be a solution of \eqref{eq: wave} arising from initial data
\begin{equation}
    \vu(x,0):=(u(x),\dot{u}(x))\in \Hi\label{rem: eq1}
\end{equation}
and moreover~$\vu$ satisfies Assumptions \ref{asp: global} and \ref{asp: N}. 

Then, the following hold
\begin{enumerate}
\item For all~$\alpha\in (0, 1-1/\sigma)$ the free channel wave operator acting on $\vec u(0)$, see Definition~\ref{def: subsecL free channel, def 1}, namely
\eq\label{ex: omegaalpha}
\Omega_{\alpha}^*\vu(0) := s\text{-}\lim_{t\to \infty} F_\alpha(x,p,t) U_0(0,t) \vu(t),
\eeq
exists in $\Hi$. Furthermore, we have that $\Omega_\alpha^*\vec{u}(0)$ is independent on $\alpha$; specifically, for any $\alpha, \alpha'\in (0, 1-1/\sigma)$, 
\eq\label{local: weak: eq}
\Omega_\alpha^*\vec{u}(0)=\Omega_{\alpha'}^*\vec{u}(0);
\eeq
\item The solution $\vec{u}(t)$ can be asymptotically decomposed as follows: 
\eq
\|\vec{u}(t)-U_{0}(t,0)\Omega_\alpha^*\vec{u}(0)-\vec{u}_{lc}(t)\|_{\mathcal{H}}\to 0\qquad\text{ as }t\to \infty,\label{decom: goal}
\eeq
where $U_{0}(t,0)$ is the solution operator for the free system and the localised part $\vec{u}_{lc}(t)$ satisfies the estimate for any $\delta\in (0,\sigma-1)$ 
\eq
\| \mathcal{X}^{\delta} \vu_{lc}(x,t)\|_{\mathcal H}\lesssim_{E,n/2-\sigma,\sigma,n,\sigma-1-\delta} 1\label{wlc: chiN},
\eeq
\end{enumerate}
where for the operator~$\mathcal{X}$ see~\eqref{def: chi}. 
\end{theorem}

\begin{remark}
Note that inequality \eqref{wlc: chiN} can also be equivalently rephrased as follows 
\eq
\| \langle x\rangle^{\delta} |p| u_{lc,1}(t)\|+\| \langle x\rangle^{\delta}  u_{lc,2}(t)\| \lesssim_{E,n/2-\sigma,\sigma,n,\sigma-1-\delta}1,
\eeq
for all $t\geq 1$, where $u_{lc}(t)=(u_{lc,1}(t),u_{lc,2}(t))$.
In contrast, 
\eq
\| \langle x\rangle^{\delta} U_0(t)f\|\geq ct^{\delta}, f\in \Hi.
\eeq
\end{remark}

With the assumption that the interaction $N(x,t,u)$ is non-local, namely see Assumption~\ref{asp: N: non-local} instead of Assumption~\ref{asp: N}, then we have the existence of the free channel wave operator as well:

\begin{theorem}\label{main theorem 3}
Let~$n\geq 4$. Let $\vu(t)=(u(x,t),\dot{u}(x,t))$ be a solution of equation \eqref{eq: wave} arising from initial data
\begin{equation}
    \vu(x,0):=(u(x),\dot{u}(x))\in \Hi,
\end{equation}
and moreover the Assumptions \ref{asp: global} and \ref{asp: N: non-local} hold. Then, for all~$\alpha\in (0, \frac{n-3}{2n+1})$ the free channel wave operator acting on $\vec u(0)$, namely\par
\eq
\Omega_{\alpha}^*\vu(0) := s\text{-}\lim_{t\to \infty} F_\alpha(x,p,t) U_0(0,t) \vu(t)
\eeq
exists in $\Hi$. Furthermore, we have that $\Omega_\alpha^*\vec{u}(0)$ is independent on $\alpha$; specifically, for any $\alpha, \alpha'\in (0, \frac{n-3}{2n+1})$ we have
\eq
\Omega_\alpha^*\vec{u}(0)=\Omega_{\alpha'}^*\vec{u}(0).
\eeq
\end{theorem}

As a direct consequence of the methods developed to prove the above theorems, we also have the following

\begin{theorem}[Perturbations of flat space]\label{thm: 3}
Suppose the interaction contains terms of the form $\partial_i h ^{ij}(x,t)\partial_j$, where~$h^{ij}:\mathbb{R}^n\times \mathbb{R}\rightarrow \mathbb{R}$ are smooth bounded functions, that satisfy the following bound 
\begin{equation}\label{eq: thm: 3, eq 1}
    \| \langle x\rangle^2 h^{ij}(x,t)\|_{L^\infty_{x,t}(\mathbb R^{n+1})}<\infty,
\end{equation}
for any index~$i,j$.

Moreover, assume that 
\begin{equation}
    N(x,t,u)=N_0(x,t,u)+\partial_i h^{ij}(x,t)\partial_j
\end{equation}
where $N_0(x,t,u)$ satisfies Assumption~\ref{asp: N: non-local}.

Then, the free channel wave operator acting on $\vec{u}(0)$ given by 
\begin{equation}
     \Omega_{\alpha}^*\vec{u}(0):=s\text{-}\lim\limits_{t\to \infty} \mathcal F(x,p,t) U_0(0,t) \vec u(t)\qquad \text{ in }\mathcal H
 \end{equation}
 exists, where $\mathcal F_\alpha(x,p,t)$ is defined as follows
 \begin{equation}
     \mathcal F_\alpha(x,p,t)=\begin{pmatrix}
         |p|^{-1} F_\alpha(x,p,t) |p| & 0\\
         0 & F_\alpha(x,p,t)
     \end{pmatrix}
 \end{equation}
for all $\alpha\in (0,1/3)$ and $n\geq 4$. Note that~$F_\alpha(x,p,t):=F_1(|x|\leq t^{\alpha})F_1(|p|\leq t^{\alpha})$.

\end{theorem}

\section{Auxiliary technical results I:~Lemma~\ref{lem: decay},~Lemma~\ref{rep: ut}}\label{sec: WO}

In the present Section we prove estimates for the free wave as well as a Duhamel formula for the inhomogeneous wave.

Throughout this paper, the Fourier transform and its inversion are given by 
\begin{equation}
    \hat f(\xi):=\frac{1}{(2\pi)^{n/2}}\int e^{-ix\cdot \xi} f(x)dx, \qquad 
    f(x):=\frac{1}{(2\pi)^{n/2}}\int e^{ix\cdot \xi} \hat f(\xi)d\xi,\label{def: Fourier}
\end{equation}
respectively.

\subsection{The free wave equation}\label{subsec: sec: WO, subsec 1}

In the following lemma we taylor already known results in the language of the present paper.

\begin{lemma}\label{lem: decay}

Let~$\vu_0(t) = (u_0(t), \dot{u}_0(t))$ be a solution to the free wave equation
\eq
\begin{cases}
\Box u_0(t)=0\\
\vu_0(0)=\vu(0)=(u(x,0),\dot{u}(x,0))\in \Hi
\end{cases},\quad (x,t)\in \mathbb{R}^n\times \mathbb{R}.\label{Pfree}
\eeq
Then, the following hold
\begin{itemize}
\item[(a)] Let~$\alpha$ and~$\sigma$ be as in Theorem \ref{main theorem 1} and
\eq
F_1(|p|,t):=F_1(t^{\alpha/2}|p|\geq 1)F_1(|p|\leq t^{\alpha/2}).\label{def: F1pt}
\eeq
There exists a constant~$C=C(n,\sigma)>0$ such that we have the following decay estimate for the solution~$u_0$
\begin{align}\label{decay: free}
&\| \langle x\rangle^{-\sigma} F_1(|p|,t)|p|u_0(t+s)\|^2+\| \langle x\rangle^{-\sigma} F_1(|p|,t)\dot{u}_0(t+s)\|^2\nonumber\\
&\qquad\qquad\qquad\qquad\qquad\qquad\qquad\qquad\leq \frac{C}{|t|^{2\sigma }}\left(\| \langle x\rangle^\sigma |p|u_0(s)\|^2+\| \langle x\rangle^\sigma \dot{u}_0(s)\|^2\right)
\end{align}
for all $|t|\geq 1$ and $s\in \mathbb{R}$. 
\item[(b)] Let $\alpha$ and $F_\alpha (x,p,t)$ be as in Theorem~\ref{main theorem 3}. If $u(x,0)=0$ and $\dot u(\cdot,0)\in L^1(\mathbb{R}^n)$ then there exists a constant~$C=C(\alpha)>0$ such that we have the following estimate
\begin{align}
\|F_\alpha (x,p,t) |p|u_0(-t)\|\leq \frac{C\log_2|t|}{|t|^{(n-1)/2-(n+1/2)\alpha}} \| \dot u(x,0)\|_{L^1_x(\mathbb{R}^n)}\label{est: decay non-local: goal}
\end{align}
and
\begin{align}
\|F_\alpha (x,p,t) \dot u_0(-t)\|\leq \frac{C\log_2|t|}{|t|^{(n-1)/2-(n+1/2)\alpha}} \| \dot u(x,0)\|_{L^1_x(\mathbb{R}^n)}\label{est: decay non-local: goalD}
\end{align}
for all $|t|\geq 2$. 
\end{itemize}
\end{lemma}
\begin{remark} Note that~\eqref{decay: free} is equivalent to the following statement
\begin{align}
\| \begin{pmatrix}
    |p|^{-1} \langle x\rangle^{-\sigma} |p|F_1(|p|,t) & 0 \\
    0 & \langle x\rangle^{-\sigma} F_1(|p|,t)
\end{pmatrix}  U_0(t+s,s)\vec{u}(s)\|_{\Hi\to \Hi}\leq  \frac{C}{|t|^{\sigma}}\| \mathcal{X}^\sigma \vu(s)\|_{\Hi},\label{decay: eq1}
\end{align}
where for the operator $\mathcal{X}$ see~\eqref{def: chi}. Furthermore, we note that for~$n\geq 4$ the estimates~\eqref{est: decay non-local: goal} and~\eqref{est: decay non-local: goalD} imply the following
\begin{align}
&\| \begin{pmatrix}
    |p|^{-1}F_\alpha(x,|p|,t)|p| & 0 \\
    0 & F_\alpha(x,|p|,t)
\end{pmatrix}  U_0(0,t)\mathcal N(\vu(t))\vu(t)\|_{\Hi}\nonumber\\
&\qquad\qquad\qquad\qquad\qquad\qquad  \leq  \frac{C\log_2|t|}{|t|^{(n-1)/2-(n+1/2)\alpha}} \sup\limits_{s\geq 0}\|  N(u(s))u(s)\|_{L^1_x(\mathbb{R}^n)}\in L^1_t[2,\infty).\label{decay: eq2}
\end{align}
\end{remark}

\begin{proof}[\textbf{Proof of Lemma~\ref{lem: decay}}]
We begin with the proof of part (a). It suffices to check the case when $s=0$. Let $H_0 = -\Delta_x$. Then, we obtain that $u_0(t)$ and $\dot{u}_0(t)$ can be represented as 
\begin{align}
u_0(t) &= \cos(t\sqrt{H_0})u(0) + \frac{\sin(t\sqrt{H_0})}{\sqrt{H_0}}\dot{u}(0)\label{free: wave1}\\
\dot{u}_0(t) &= -\sin(t\sqrt{H_0})\sqrt{H_0}u(0) + \cos(t\sqrt{H_0})\dot{u}(0),\label{free: wave2}
\end{align}
respectively. 

By using the operator
\eq
A_0:=\begin{pmatrix} 0& -1\\H_0&0\end{pmatrix}.
\eeq
we can reformulate the free wave equation as a first-order system:
\eq
\p_t[\vu_0(t)]=-A_0\vu_0(t).
\eeq
We solve for~$\vu_0(t)$ and obtain the representation formula:
\eq
\vu_0(t)=e^{-tA_0}\vu(0),
\eeq
where the evolution operator $U_0(t,0)$ is given by
\eq
U_0(t,0)=e^{-tA_0}.\label{Ufree}
\eeq
and~$U_0(t,s)=U_0(t-s,0)$.

Now, consider the following conditions:
\begin{equation}
    |t|>1,\qquad  \alpha\in (0,1-1/\sigma),\qquad \sigma>1
\end{equation}
and 
\begin{equation}
    (\langle x\rangle^\sigma|p|u_0(0),\langle x\rangle^\sigma\dot{u}_0(0))\in L^2_x(\mathbb{R}^n)\times L^2_x(\mathbb{R}^n).
\end{equation}
Under these conditions, by using the method of non-stationary phase, we obtain that the free solution $\vec{u}_0(t)=(u_0(t),\dot{u}_0(t))$ satisfies the dispersive estimates:
\begin{align}
\| \langle x\rangle^{-\sigma}F_1(|p|,t) |p|u_0(t)\|^2\leq& \frac{C}{|t|^{2\sigma }}\left(\| \langle x\rangle^\sigma |p|u_0(0)\|^2+\| \langle x\rangle^\sigma \dot{u}_0(0)\|^2\right)
\end{align}
and
\begin{align}
\|\langle x\rangle^{-\sigma}F_1(|p|,t)\dot{u}_0(t) \|^2\leq& \frac{C}{|t|^{2\sigma }} \left(\| \langle x\rangle^\sigma |p|u_0(0)\|^2+\| \langle x\rangle^\sigma \dot{u}_0(0)\|^2\right),
\end{align}
where $C=C(n,\sigma)>0$. 

Next, we prove part (b). We use~(2.4) from~\cite{S2021} in conjunction with that $|p|$ commutes with the free evolution operator $U_0(t,0)$ to obtain that $|p|F_1(|p|\leq t^\alpha)u_0(t)$ satisfies the following dispersive decay
\eq\label{est: decay non-local}
\| |p|F_1(|p|\leq t^\alpha)u_0(t)\|_{L^\infty_x(\mathbb{R}^n)}\lesssim t^{-\frac{n-1}{2}}\| |p|F_1(|p|\leq t^\alpha) \dot u_0(x,0)\|_{\dot B_{1,1}^{\frac{n-1}{2}}},\qquad t\geq 1,
\eeq
where $\dot B^\alpha_{1,1}$ stands for the usual Besov space: $\| f\|_{\dot B^\alpha_{1,1}}=\sum_{j\in \mathbb{Z}} 2^{\alpha j }\|p_j f\|_{L^1_x(\mathbb{R}^n)}$ where $P_j$ is the Littlewood Paley projection onto frequencies of size $2^j$. By using the cut-off definitions of Section~\ref{subsecL free channel} and moreover by using
\eq
\|p_j |p|F_1(|p|\leq t^\alpha) \|_{L^1_x(\mathbb{R}^n)\to L^1_x(\mathbb{R}^n)}\lesssim t^\alpha,
\eeq
we obtain 
\begin{align}\label{eq: proof: lem: decay, eq 1}
    \| |p|F_1(|p|\leq t^\alpha) \dot u_0(x,0)\|_{\dot B_{1,1}^{\frac{n-1}{2}}}= & \sum_{j\in \mathbb Z} 2^{(n-1)j/2 }\|P_j |p|F_1(|p|\leq t^\alpha) \dot u_0(x,0)\|_{L^1_x(\mathbb{R}^n)}\nonumber\\
    \lesssim & \alpha(\log_2 t )t^{(n-1)\alpha/2} \|p_j |p|F_1(|p|\leq t^\alpha) \dot u_0(x,0)\|_{L^1_x(\mathbb{R}^n)}\nonumber\\
    \lesssim & \alpha(\log_2 t )t^{(n+1)\alpha/2} \|\dot u_0(x,0)\|_{L^1_x(\mathbb{R}^n)},
\end{align}
for $t\geq 2$. Note that~\eqref{eq: proof: lem: decay, eq 1} together with~\eqref{est: decay non-local} and the estimate $\|F_c(\frac{|x|}{t^\alpha}\leq 1)\|_{L^\infty_x(\mathbb R^n)\to L^2_x(\mathbb R^n)}\lesssim t^{n\alpha/2}$ concludes the desired~\eqref{est: decay non-local: goal}. Similarly, we have~\eqref{est: decay non-local: goalD}.
\end{proof}

We have similar free estimates for $e^{\pm i t|p|}=e^{\pm it\sqrt{H_0}}, t\geq 0$: 
\begin{lemma}For all $i=1,\cdots,n$, $n\geq 1$ and $\beta\in (0, 1)$,
\begin{equation}
    \|\langle x\rangle^{-2}e^{\pm it|p|}F_1(|p|\leq t^\beta)\p_i\langle x\rangle^{-2}\|\lesssim_{ n} \frac{1}{\langle t\rangle^{2-\beta}}.\label{decay ineq}
\end{equation}
\end{lemma}
\begin{proof} Let $\mathscr O(t):=\langle x\rangle^{-2}e^{\pm it|p|}F_1(|p|\leq t^\beta)\p_i\langle x\rangle^{-2}$. We break $\mathscr O(t)$ into three pieces
\begin{equation}
    \mathscr O(t)=\sum\limits_{j=1}^3  \mathscr O_j(t),\label{eq: decom O}
\end{equation}
where the operators $\mathscr O_j(t), j=1,2,3,$ are given by 
\begin{equation}
    \mathscr O_1(t):=\langle x\rangle^{-2}e^{\pm it|p|}F_1(|p|\leq t^\beta)\p_i\langle x\rangle^{-2}\chi(|x|> \frac{1}{10}\langle t\rangle),
\end{equation}
\begin{equation}
    \mathscr O_2(t):=\chi(|x|> \frac{1}{10}\langle t\rangle)\langle x\rangle^{-2}e^{\pm it|p|}F_1(|p|\leq t^\beta)\p_i\langle x\rangle^{-2}\chi(|x|\leq \frac{1}{10}\langle t\rangle)
\end{equation}
and 
\begin{equation}
    \mathscr O_3(t):=\chi(|x|\leq \frac{1}{10}\langle t\rangle)\langle x\rangle^{-2}e^{\pm it|p|}F_1(|p|\leq t^\beta)\p_i\langle x\rangle^{-2}\chi(|x|\leq \frac{1}{10}\langle t\rangle).
\end{equation}
$\mathscr O_j(t), j=1,2,$ satisfy 
\begin{equation}
    \|\mathscr O_j(t)\|\lesssim \frac{1}{\langle t\rangle^{2-\beta}}, \qquad j=1,2.\label{est: O12}
\end{equation}
To estimate~$\mathscr O_3(t)$, we use the Fourier transform and the Fourier inversion theorem to integrate by parts twice in $|p|$ in the Fourier space in order to get 
\begin{equation}
    \|\mathscr O_3(t)\|\lesssim \frac{1}{\langle t\rangle^{2-\beta}}.\label{est: O3}
\end{equation}
Estimates~\eqref{est: O12} and~\eqref{est: O3} and Eq.~\eqref{eq: decom O} yield~\eqref{decay ineq}. 

\end{proof}

\subsection{Duhamel formulas}\label{subsec: sec: WO, subsec 2}

We outline some fundamental properties of the solutions to the perturbed wave equation.
\begin{lemma}\label{rep: ut}
Let $(u(t),\dot{u}(t))$ be a solution to the perturbed wave equation
\eq
\begin{cases}
\Box u(t)=-V(x,t)u(t)\\
\vu(0)=(u(x,0),\dot{u}(x,0))\in \Hi
\end{cases},\quad (x,t)\in \mathbb{R}^n\times \mathbb{R}.\label{Ppert}
\eeq
Then, the solution $\vec{u}(t)$ can be represented by the following formula:
\begin{equation}
    \vec{u}(t)=U_0(t,0)\vec{u}(0)-\int_0^t U_0(t,s) \mathcal{V}(x,s) \vec{u}(s)ds,
\end{equation}
where the matrix $\V(x,t)$ is defined by
\eq
\V(x,t):=\begin{pmatrix} 0&0\\ V(x,t)& 0\end{pmatrix},
\eeq
~$U_0(t,0)= e^{-t A_0}$ and~$U_0(t,s)=U_0(t-s,0)$.

\end{lemma}

\begin{proof}
We can reformulate Equation \eqref{Ppert} as follows
\eq
\p_t[\vu(t)]=-(A_0+\V(x,t))\vu(t),
\eeq
or equivalently as follows
\eq
\p_t[ U(t,0)\vu(0)]=-(A_0+\V(x,t))U(t,0)\vu(0).\label{Upert}
\eeq
Utilizing Equations \eqref{Ufree} and \eqref{Upert}, we can derive Duhamel's formula for $\vu(t)$:
\begin{align}
\vu(t) &= U_0(t,0)\vu(0) - \int_0^t  U_0(t,s)\V(x,s)\vu(s)ds \label{D1}\\
&= \begin{pmatrix} u(t) \\ \dot{u}(t) \end{pmatrix}.\nonumber
\end{align}

\end{proof}
Similarly, using Equation \eqref{D1}, we derive Duhamel's formula for $\Omega(t)^*\vu(0)\equiv U_0(0,t)\vu(t)$: 
\begin{align}
\Omega(t)^*\vu(0) &= \vu(0) - \int_0^t  U_0(0,s)\V(x,s)\vu(s) ds.
\end{align}
Consequently, with $\Omega(t)^*\vu(0)\equiv(u_\Omega(t), \dot u_\Omega(t))$, the expressions for $u_\Omega(t)$ and $\dot{u}_\Omega(t)$ are given by: 
\begin{align}
u_\Omega(t) &= u(0) + \int_0^t  \frac{\sin{(s\sqrt{H_0})}}{\sqrt{H_0}}V(s)u(s) ds\label{d_omegau},\\
\dot{u}_\Omega(t) &= \dot{u}(0) - \int_0^t  \cos(s\sqrt{H_0})V(s)u(s) ds.\label{omegau}
\end{align}

\section{Auxiliary technical results II:~(Relative)-Propagation estimates}\label{subsec: sec: WO, subsec 3}

The present Section gathers several technical results we will need later in Section~\ref{sec: proof of thm1}, where we prove our main Theorem. 

\subsection{Preliminary Commutator estimate}

We have the following commutator estimate.

\begin{lemma}\label{com}

Define~$F_1^{l}(k):=\frac{d^lF_1}{dk^l},$ where~$l=0,1$. For~$t\geq 1$,~$\alpha>0$ and~$ \alpha>\beta$, the following inequality holds for all~$n \geq 1$:
\eq
\|[F_c( \frac{|x|}{t^\alpha}\leq 1),F_1^{(l)}(t^{\beta}|p|>1)]\|\lesssim_n \frac{1}{t^{\alpha-\beta}}.
\eeq
\end{lemma}
\begin{proof}
In this proof, we write $F_c=F_c(\frac{|x|}{t^\alpha}\leq 1)$ and $F_1^{(l)}=F_1^{(l)}(t^\beta |p|>1)$ for simplicity. We use the Fourier transform and write~$[  F_c, F_1^{(l)}]$ as follows
\begin{equation}
    \begin{aligned}
        [  F_c, F_1^{(l)}]&=
\frac{1}{(2\pi)^{n/2}}\int d^n\xi \hat{F}_1^{(l)}(\xi)e^{it^{b}p\cdot \xi}\times\left[e^{-it^{\beta}p\cdot \xi}F_c(\frac{|x|}{t^\alpha}\leq 1) e^{it^{\beta}p\cdot \xi}-F_c(\frac{|x|}{t^\alpha}\leq 1) \right]\\
&=\frac{1}{(2\pi)^{n/2}}\int d^n\xi \hat{F}_1^{(l)}(\xi) e^{it^{\beta}p\cdot \xi}(F_c(\frac{|x-t^{\beta}\xi|}{t^\alpha}\leq 1)-F_c(\frac{|x|}{t^\alpha}\leq 1)).
    \end{aligned}
\end{equation}
Given that
\eq
\dfrac{\left|F_c(\frac{|x-t^{\beta}\xi|}{t^\alpha}\leq 1)-F_c(\frac{|x|}{t^\alpha}\leq 1) \right|}{t^{\beta-\alpha}|\xi|}\lesssim \sup\limits_{x\in \mathbb{R}^n}| F_c'(|x|\leq 1)|\lesssim 1,
\eeq
we can deduce that for each~$\psi \in L^2_x(\mathbb{R}^n)$ we obtain
\begin{equation}
\| [F_c( \frac{|x|}{t^\alpha}\leq 1)), F_1(t^{\beta}|p|>1)]\psi\|\leq\frac{1}{t^{\alpha-\beta}}\int d^n\xi | \hat{F}_1^{(l)}(\xi)|\xi| \|\psi(x)\|\lesssim \frac{1}{t^{\alpha-\beta}}\|\psi\|,
\end{equation}
which concludes the proof.
\end{proof}

\subsection{Definitions of (Relative)~Propagation observables}

We define the following propagation observables, inspired by~\cite{SW20221}.

\begin{definition}\label{def: subsec: sec: WO, subsec 3, def 1}(\textbf{Propagation observable})
Given a class of matrix operators $\{B(t)\}_{t\geq 0}$ with 
 \eq
 B(t)=\begin{pmatrix}
     B_1(t)& 0 \\
     0& B_2(t)
 \end{pmatrix},
 \eeq
 we define the time-dependent scalar product as follows 
 \begin{align}
    \langle B(t), \vec{u}(t)\rangle_t:=(|p| u_1(t),B_1(t)|p|u_1(t))_{L^2_x(\mathbb{R}^n)}+( u_2(t),B_2(t)u_2(t))_{L^2_x(\mathbb{R}^n)},
\end{align}
where $\vec{u}(t)$ denotes the solution to \eqref{eq: wave}. The family $\{B(t)\}_{t\geq 0}$ is termed a Propagation Observable if it satisfies the following condition: For a family of self-adjoint operators $B(t)$, the time derivative satisfies: there exists $L\in \mathbb{N}^+$ such that
\begin{align}\label{eq: def: subsec: sec: WO, subsec 3, def 1, eq 1}
    & \partial_t \langle B(t), \vec{u}(t)\rangle_t=\pm\sum\limits_{l=1}^L(\vec{u}(t),C_l^*(t)C_l(t)\vec{u}(t))_{\mathcal{H}}+g(t)\\
    & g(t)\in L^1_{t}[1,\infty),\quad C_l^*(t)C_l(t)\geq0, \quad l=1,\cdots, L.\nonumber
\end{align}
\end{definition}

Integrating~\eqref{eq: def: subsec: sec: WO, subsec 3, def 1, eq 1} over time, we derive the \underline{Propagation Estimate}: 
\begin{align}
   \sum\limits_{l=1}^L \int_{t_0}^T\|C_l(t) \vec{u}(t) \|_{\mathcal{H}}^2dt&= \pm\langle B(t), \vec{u}(t)\rangle_t\vert_{t=t_0}^{t=T}\mp\int_{t_0}^Tg(s) ds\nonumber\\
\leq& \sup\limits_{t\geq t_0} \left|\langle B(t),\vec{u}(t)\rangle_t\right|+C_g,  
\end{align}
where $C_g:=\|g(t)\|_{L^1_t[1,\infty)}.$ 

Second, we define the relative propagation observables.

\begin{definition}\label{def: subsec: sec: WO, subsec 3, def 2}(\textbf{Relative Propagation observables})
Consider a class of matrix operators $\{ \tilde{B}(t)\}_{t\geq 0}$  with 
 \eq
\tilde{B}(t)=\begin{pmatrix}
     \tilde{B}_1(t)& 0 \\
     0& \tilde{B}_2(t)
 \end{pmatrix}.
 \eeq 
 We denote their time-dependent expectation values as:
 \begin{align}
     \langle \tilde{B}: \vec{v}(t)\rangle_t:=(|p|v_1(t), \tilde{B}_1(t)|p|v_1(t)  )_{L^2_x(\mathbb{R}^n)}+(v_2(t), \tilde{B}_2(t)v_2(t)  )_{L^2_x(\mathbb{R}^n)},
 \end{align}
 where $\vec{v}(t)$ is not necessarily the solution to \eqref{eq: wave}, but satisfies the condition:
\eq
\sup\limits_{t\geq 0}\langle \tilde{B}:\vec{v}(t)\rangle_t<\infty. \label{phiH}
\eeq
Let \eqref{phiH} hold, and moreover let the time derivative $\partial_t\langle \tilde{B}:\vec{v}(t)\rangle_t$ satisfy the following: there exists $L\in \mathbb{N}^+$ such that
\begin{align}\label{eq: def: subsec: sec: WO, subsec 3, def 2, eq 1}
&\partial_t\langle \tilde{B} : \vec{v}(t)\rangle_t=\pm \sum\limits_{l=1}^L ( \vec{v}(t), C^*_l(t)C_l(t)\vec{v}(t))_{\Hi}+g(t)\\
&g(t)\in L^1[1,\infty), \quad C_l^*C_l\geq 0 , \quad l=1,\cdots,L.\nonumber
\end{align}
Then, the family $\{\tilde{B}(t)\}_{t\geq 0}$ is termed as a Relative Propagation Observable with respect to $\vec{v}(t)$.   
\end{definition}

Integrating~\eqref{eq: def: subsec: sec: WO, subsec 3, def 2, eq 1} over time we derive the \underline{Relative Propagation Estimate}. 
\begin{equation}\label{CC}
    \begin{aligned}
        \sum\limits_{l=1}^L\int_{t_0}^T\|C_l(t) \vec{v}(t) \|_{\mathcal{H}}^2dt&=\pm\langle \tilde{B}(t): \vec{v}(t)\rangle_t\vert_{t=t_0}^{t=T}\mp\int_{t_0}^Tg(s) ds\\
& \leq \sup\limits_{t\geq t_0} \left|\langle \tilde{B}(t):\vec{v}(t)\rangle_t\right|+C_g,
    \end{aligned}
\end{equation}
 where~$C_g=\|g(t)\|_{L^1_t[1,\infty)}$.

We have the following remark.

\begin{remark}
Upon identifying a relative propagation observable where $C_l(t)$ can function either as a multiplication operator or as a multiplication operator in Fourier space, we then apply H\"older's inequality on each component. This application leads to the following implication from \eqref{CC}: for all $T\geq t_0\geq 1$,
\eq
\| \int_{t_0}^T dt C_l^*(t)C_l(t)\vec{v}(t)\|_{\Hi}\leq \left(\int_{t_0}^T dt\| |C_l(t)|^2\|_{\Hi\to \Hi} \right)^{1/2}\left(\int_{t_0}^T\|C_l(t)\vec{v}(t) \|_{\Hi}^2dt\right)^{1/2} \to 0,
\eeq
as $t_0\to \infty$. 
\end{remark}

Now, note the following definition.

\begin{definition}\label{def: subsec: sec: WO, subsec 3, def 3}
We define the operators
\eq
B_1(t)=\begin{pmatrix} |p|^{-1}F_1(|p|,t)F_\alpha(x,p,t)|p| & 0\\
0& F_1(|p|,t)F_\alpha(x,p,t)
\end{pmatrix},\label{B1}
\eeq
\eq
B_2(t)=\begin{pmatrix}
   |p|^{-1} F_\alpha(x,p,t)F_c(\frac{|x|}{t^\alpha }\leq 1)|p| & 0\\
   0 & F_\alpha(x,p,t)F_c(\frac{|x|}{t^\alpha }\leq 1)
\end{pmatrix},\label{B2}
\eeq
where recall that $F_\alpha(x,p,t)$ is defined in~\eqref{Falpha1} and 
\eq
F_1(|p|,t):=F_1(t^{\alpha/2}|p|\geq 1)F_1(|p|\leq t^{\alpha/2}).
\eeq
Moreover we define the operators
\eq
B_3(t)=\begin{pmatrix}
    |p|^{-1}F_1(|p|\leq t^{\alpha}) F_\alpha(x,p,t) |p| & 0 \\
    0 &F_1(|p|\leq t^{\alpha})F_\alpha(x,p,t)
\end{pmatrix}
\eeq
and
\eq
B_4(t)=\begin{pmatrix}
    |p|^{-1}F_\alpha(x,p,t) F_c(\frac{|x|}{t^\alpha}\leq 1) |p| & 0 \\
    0 &F_\alpha(x,p,t)F_c(\frac{|x|}{t^\alpha}\leq 1)
\end{pmatrix},
\eeq
where $F_\alpha$ is defined by~\eqref{Falpha2}. 
\end{definition}

\begin{remark}
    The operators~$B_1,B_2$ will be used in the case when the interaction~$N$ is local, see Assumption~\ref{asp: N}, while the operators~$B_3,B_4$ will be used when the interaciton is non-local, see Assumption~\ref{asp: N: non-local}. 
\end{remark}

Note the following lemmata. 

\subsection{The technical Lemmata~\ref{lem: 4.2},~\ref{lem: prop: B1B2},~\ref{lem: propest: B1B2}}

\begin{lemma}\label{lem: 4.2} Let~$u$ be a solution of~\eqref{eq: wave} and let Assumption \ref{asp: global} hold.
Then, for~$B_j$ as in Definition~\ref{def: subsec: sec: WO, subsec 3, def 3} the following holds
\eq
\sup\limits_{t\geq 0} \langle B_j: U_0(0,t)\vec{u}(t)\rangle_t\lesssim_{E}1,\quad j=1,2,3,4.\label{uniform: B}
\eeq
for $j=1,2,3,4.$
    
\end{lemma}
\begin{proof}From Assumption~\ref{asp: global}, along with the fact that $\sup\limits_{t\in \mathbb{R}}\|U_0(t,0)\|_{\Hi\to \Hi}<\infty$, \eqref{uniform: B} follows.     
\end{proof}

Now, we are ready to prove the following

\begin{lemma}\label{lem: prop: B1B2}
Let~$u$ be a solution of~\eqref{eq: wave}. 

$\bullet$ If Assumptions~\ref{asp: global} and~\ref{asp: N} hold, then both $\{B_1(t)\}_{t\geq 0}$ and $\{B_2(t)\}_{t\geq 0}$ are relative propagation observable with respect to $\vec{v}(t)=U_0(0,t)\vec{u}(t)$.

$\bullet$ If Assumptions~\ref{asp: global} and~\ref{asp: N: non-local} hold, then both $\{B_3(t)\}_{t\geq 0}$ and $\{B_4(t)\}_{t\geq 0}$ are relative propagation observable with respect to $\vec{v}(t)=U_0(0,t)\vec{u}(t)$.
    
\end{lemma}
\begin{proof}
We prove that $\{B_1(t)\}_{t\geq 0}$ is relative propagation observable with respect to $\vec{v}(t)$. $\{B_j(t)\}_{t\geq 0},j=2,3,4,$ are treated similarly. For brevity, we denote
\begin{equation}
    F_c(\frac{|x|}{t^\alpha}\leq 1),\qquad F_1(|p|,t)
\end{equation}
as 
\begin{equation}
    F_c,\qquad F_1
\end{equation}
respectively. 

We compute 
\begin{align}
    \partial_t\langle B_1(t): \vec{v}(t)\rangle_t=& \langle \partial_t[B_1(t)]: \vec{v}(t)\rangle_t+(\vec{v}(t), B_1(t)U_0(0,t)\mathcal{N}(x,t,\vec{u})\vec{u}(t) )_{\Hi}\nonumber\\
    &+(B_1(t)U_0(0,t)\mathcal{N}(x,t,\vec{u})\vec{u}(t) , \vec{v}(t))_{\Hi}\nonumber\\
    =:& a_p(t)+a_{in1}(t)+a_{in2}(t).
\end{align}
We write $B_1(t)$ as follows 
\eq
B_1(t)=\tilde{B}_{11}(t)\tilde{B}_{12}(t),
\eeq
where $\tilde{B}_{11}(t)$ and $\tilde{B}_{12}(t)$ are defined by 
\eq
\tilde{B}_{11}(t):=\begin{pmatrix}
|p|^{-1}F_1F_c\langle x\rangle^\sigma |p| & 0\\
0& F_1\langle x\rangle^\sigma F_c
\end{pmatrix},
\qquad\qquad
\tilde{B}_{12}(t):=\begin{pmatrix}
    |p|^{-1} \langle x\rangle^{-\sigma} |p|F_1 & 0 \\
    0 & \langle x\rangle^{-\sigma} F_1
\end{pmatrix}
\eeq
respectively. We obtain that $a_{in1}(t), a_{in2}(t)\in L_t^1[1,\infty)$ since according to Lemma \ref{lem: decay} and Equation \eqref{decay: eq1}, we have for $k=1,2,$ the following
\begin{equation}
    \begin{aligned}
        |a_{ink}(t) | \leq & \left(\sup\limits_{t\geq 0}\| \vec{v}(t) \|_{\Hi}\right)\times  \| B_1(t)U_0(0,t) \mathcal{N}(x,t,\vu)\vu(t)\|_{\Hi} \\
\leq &\left(\sup\limits_{t\geq 0}\| \vec{v}(t) \|_{\Hi}\right)\times |t|^{\sigma \alpha} \| \tilde{B}_{12}(t)U_0(0,t)\mathcal{N}(x,t,\vu)\vu(t)\|_{\Hi}\\
\leq & \left(\sup\limits_{t\geq 0}\| \vec{v}(t) \|_{\Hi}\right)\times \frac{C}{|t|^{\sigma(1-\alpha)}}\sup\limits_{t\geq 0} \|  \mathcal X^\sigma \mathcal{N}(x,t,\vu)\vu(t)\|_{\Hi}\\
    \end{aligned}
\end{equation}
which concludes that~$|a_{ink}(t)| \in  L^1_t[1,\infty)$
provided that $\sigma>1$ and $\alpha\in (0, 1-1/\sigma)$.

We write 
\begin{align}
    \partial_t[B_1(t)]=& B_{1p1}(t)+B_{1p2}(t)+B_{1r1}(t)+B_{1r2}(t)
\end{align}
where $B_{1p1}(t),B_{1p2}(t),$ $B_{1r1}(t)$ and $B_{1r2}(t)$ are given by 
\begin{equation}\label{eq4.31}
    \begin{aligned}
        B_{1p1}(t)  &   =\begin{pmatrix}
|p|^{-1}F_1\partial_t[F_c]F_1 |p| & 0\\
0& F_1 \partial_t[F_c] F_1
\end{pmatrix},\\
B_{1p2}(t)  & =\begin{pmatrix}
2|p|^{-1}\sqrt{F_c}F_1\partial_t[F_1] \sqrt{F_c}|p| & 0\\
0& 2\sqrt{F_c}F_1 \partial_t[F_1]\sqrt{F_c}
\end{pmatrix},\\
B_{1r1}(t)  &   =\begin{pmatrix}
|p|^{-1}F_1\sqrt{F_c}[\sqrt{F_c},\partial_t[F_1]] |p| & 0\\
0& F_1 \sqrt{F_c}[\sqrt{F_c}],\partial_t[F_1]]
\end{pmatrix},\\
B_{1r2}(t)  &   =\begin{pmatrix}
|p|^{-1}[F_1,\sqrt{F_c}]\partial_t[F_1]\sqrt{F_c}|p| & 0\\
0& [F_1, \sqrt{F_c}]\partial_t[F_1]\sqrt{F_c}
\end{pmatrix},
    \end{aligned}
\end{equation}
respectively. 

Now, since we have 
\eq
\p_t[F_c]=F_c'(\frac{|x|}{t^\alpha}\leq 1) \times \frac{-|x|}{t^{1+\alpha}}\geq 0
\eeq
and moreover the following holds in Fourier space 
\begin{align}
\p_t[F_1(|p|,t)]=&F_1'(t^{\alpha/2}|p|\geq 1) F_1(|p|\leq t^{\alpha/2})\times t^{\alpha/2}\nonumber\\
&+F_1(t^{\alpha/2}|p|\geq 1) F_1'(|p|\leq t^{\alpha/2})\times \frac{-|p|}{t^{1+\alpha/2}}\geq 0,
\end{align}
we obtain the following
\begin{align}
\langle B_{1p1}(t): \vec{v}(t) \rangle_t =& (|p|v_1(t), F_1 \partial_t[F_c]F_1|p|v_1(t))_{L^2_x(\mathbb{R}^n)} +  (v_2(t),  F_1\partial_t[F_c]F_1v_2(t))_{L^2_x(\mathbb{R}^n)}\nonumber\\
\geq& 0 \label{B1p1t}
\end{align}
and 
\begin{align}
\langle B_{1p2}(t): \vec{v}(t) \rangle_t =& 2(|p|v_1(t), \sqrt{F_c}F_1\partial_t[F_1] \sqrt{F_c}|p|v_1(t))_{L^2_x(\mathbb{R}^n)}\nonumber\\
&+  2(v_2(t),  \sqrt{F_c}F_1\partial_t[F_1] \sqrt{F_c}v_2(t))_{L^2_x(\mathbb{R}^n)}\nonumber\\
\geq& 0 .\label{B1p2t}
\end{align}
We conclude that 
\begin{equation}
    \langle B_{1r1}: \vec{v}(t)\rangle_t, \langle B_{1r2}: \vec{v}(t)\rangle_t \in L^1_t[1,\infty) 
\end{equation}
by Lemma \ref{com}. Thus, $\partial_t\langle B_1(t): \vec v(t) \rangle_t$ takes the form of~\eqref{eq: def: subsec: sec: WO, subsec 3, def 2, eq 1} for all $t\geq 1$. This together with Lemma~\ref{lem: 4.2} yields that $\{B_1(t)\}_{t\geq 0}$ is a relative propagation observable with respect to $\vec{v}(t)$. Similarly, $\{B_j(t)\}_{t\geq 0}, j=2,3,4,$ are also relative propagation observables with respect to $\vec{v}(t)$.

Indeed, for~$\boxed{j=2,4}$ we exchange the roles of $F_1$ and $F_c$ with
\begin{equation}
    F_1=F_1(|p|,t),\qquad F_1=F_1(|p|\leq t^\alpha)
\end{equation}
respectively to obtain 
\begin{equation}
     \partial_t\langle B_j(t): \vec{v}(t)\rangle_t= \sum\limits_{k=1}^2\langle B_{jpk}(t): \vec{v}(t)\rangle_t+g(t),
\end{equation}
where $\langle B_{jpk}(t): \vec{v}(t)\rangle_t, k=1,2,$ and $g(t)$ are defined as follows
\begin{equation}\label{eq: proof: lem: prop: B1B2, eq 10}
    \begin{aligned}
        \langle B_{jp1}(t): \vec{v}(t) \rangle_t & = (|p|v_1(t), F_c \partial_t[F_1]F_c|p|v_1(t))_{L^2_x(\mathbb{R}^n)} \\
        &   \qquad\qquad\qquad\qquad +  (v_2(t),  F_c\partial_t[F_1]F_cv_2(t))_{L^2_x(\mathbb{R}^n)}\geq 0,\\\vspace{3mm}
        \langle B_{jp2}(t): \vec{v}(t) \rangle_t &  = 2(|p|v_1(t), \sqrt{F_1}F_c\partial_t[F_c] \sqrt{F_1}|p|v_1(t))_{L^2_x(\mathbb{R}^n)}\\
        &   \qquad\qquad\qquad\qquad + 2(v_2(t),  \sqrt{F_1}F_c\partial_t[F_c] \sqrt{F_1}v_2(t))_{L^2_x(\mathbb{R}^n)}\geq 0\\
        g(t)    &   =(\vec{v}(t), B_j(t)U_0(0,t)\mathcal{N}(\vec{u})\vec{u}(t) )_{\Hi}+(B_j(t)U_0(0,t)\mathcal{N}(\vec{u})\vec{u}(t) , \vec{v}(t))_{\Hi}\\
        &\qquad\qquad\qquad\qquad +\langle  B_{jr1}(t): \vec{v}(t)\rangle_t+\langle  B_{jr2}(t): \vec{v}(t)\rangle_t,
    \end{aligned}
\end{equation}
respectively, and~$B_{jrk}(t),~ k=1,2,$ are given by 
\eq
B_{jr1}(t)=\begin{pmatrix}
|p|^{-1}F_c\sqrt{F_1}[\sqrt{F_1},\partial_t[F_c]] |p| & 0\\
0& F_c \sqrt{F_1}[\sqrt{F_1}],\partial_t[F_c]]
\end{pmatrix}
\eeq
and 
\eq
B_{jr2}(t)=\begin{pmatrix}
|p|^{-1}[F_c,\sqrt{F_1}]\partial_t[F_c]\sqrt{F_1}|p| & 0\\
0& [F_c, \sqrt{F_1}]\partial_t[F_c]\sqrt{F_1}
\end{pmatrix},
\eeq
respectively. By a similar argument and using estimate~\eqref{decay: eq1} for $j=2$ and estimate~\eqref{decay: eq2} for $j=4$, we have $g\in L^1_t[1,\infty)$.

When $\boxed{j=3}$, we obtain
\begin{equation}
     \partial_t\langle B_j(t): \vec{v}(t)\rangle_t= \sum\limits_{k=1}^2\langle B_{jpk}(t): \vec{v}(t)\rangle_t+g(t),
\end{equation}
where $\langle B_{jpk}(t): \vec{v}(t)\rangle_t, k=1,2,$ and $g(t)$ are given by 
\begin{equation}\label{eq4.43}
    \begin{aligned}
        \langle B_{jp1}(t): \vec{v}(t) \rangle_t &  = (|p|v_1(t), F_1 \partial_t[F_c]F_1|p|v_1(t))_{L^2_x(\mathbb{R}^n)} \\
        &   \qquad\qquad\qquad\qquad+  (v_2(t),  F_1\partial_t[F_c]F_1v_2(t))_{L^2_x(\mathbb{R}^n)}\geq 0\\
        \langle B_{jp2}(t): \vec{v}(t) \rangle_t &  = 2(|p|v_1(t), \sqrt{F_c}F_1\partial_t[F_1] \sqrt{F_c}|p|v_1(t))_{L^2_x(\mathbb{R}^n)}\\
        &\qquad\qquad\qquad\qquad +  2(v_2(t),  \sqrt{F_c}F_1\partial_t[F_1] \sqrt{F_c}v_2(t))_{L^2_x(\mathbb{R}^n)}\geq 0 \\
        g(t)&   =(\vec{v}(t), B_j(t)U_0(0,t)\mathcal{N}(\vec{u})\vec{u}(t) )_{\Hi}+(B_j(t)U_0(0,t)\mathcal{N}(\vec{u})\vec{u}(t) , \vec{v}(t))_{\Hi}\\
        &\qquad\qquad+\langle  B_{jr1}(t): \vec{v}(t)\rangle_t+\langle  B_{jr2}(t): \vec{v}(t)\rangle_t,
    \end{aligned}    
\end{equation}
respectively, where~$B_{jrk}(t), k=1,2,$ are given by 
\eq
B_{jr1}(t)=\begin{pmatrix}
|p|^{-1}F_1\sqrt{F_c}[\sqrt{F_c},\partial_t[F_1]] |p| & 0\\
0& F_1 \sqrt{F_c}[\sqrt{F_c}],\partial_t[F_1]]
\end{pmatrix}
\eeq
and 
\eq
B_{jr2}(t)=\begin{pmatrix}
|p|^{-1}[F_1,\sqrt{F_c}]\partial_t[F_1]\sqrt{F_c}|p| & 0\\
0& [F_1, \sqrt{F_c}]\partial_t[F_1]\sqrt{F_c}
\end{pmatrix},
\eeq
respectively. By a similar argument, we have $g\in L^1_t[1,\infty)$ for the case~$j=3$.

\end{proof}

Note the following lemma.

\begin{lemma}\label{lem: propest: B1B2} Let~$u$ be a solution of~\eqref{eq: wave}. Let~$B_{jpk}$ be as in~\eqref{eq4.31},~\eqref{eq: proof: lem: prop: B1B2, eq 10} and~\eqref{eq4.43}. Then, with $\vec{v}(t)=U_0(0,t)\vec{u}(t)$, the following holds
\begin{equation}
    \langle B_{jpk}(t): \vec{v}(t) \rangle_t \in L^1_t[1,\infty), \qquad j=1,2,3,4,\,\, k=1,2.\label{prop: est: B1B2}
\end{equation}
\end{lemma}

\begin{proof} Let $a_{ink}(t), \langle B_{1rk}(t): \vec{v}(t)\rangle_t, k=1,2,$ be as in the proof of Lemma~\ref{lem: prop: B1B2}. By taking
\begin{equation}
   \begin{aligned}
       C_1(t)   &   =\begin{pmatrix}
|p|^{-1}\sqrt{F_1\partial_t[F_c]F_1} |p| & 0\\
0& \sqrt{F_1 \partial_t[F_c] F_1}
\end{pmatrix},\\
    C_2(t)  &   =\begin{pmatrix}
|p|^{-1}\sqrt{2\sqrt{F_c}F_1\partial_t[F_1] \sqrt{F_c}}|p| & 0\\
0& \sqrt{2\sqrt{F_c}F_1 \partial_t[F_1]\sqrt{F_c}}
\end{pmatrix},\\
g(t)    &   =a_{in1}(t)+a_{in2}(t)+\langle B_{1r1}(t): \vec{v}(t)\rangle_t+ \langle B_{1r2}(t): \vec{v}(t)\rangle_t,
   \end{aligned} 
\end{equation}
and by using Lemma~\ref{lem: prop: B1B2}, equation~\eqref{eq: def: subsec: sec: WO, subsec 3, def 2, eq 1} and~\eqref{CC} we conclude.
\end{proof}

\section{Forward/backward propagation waves: Proposition~\ref{prop: urj}}

The main result of this section is Proposition~\ref{prop: urj}, which provides estimates for the key component of the non-free part of $\vec u$, specifically $u^\pm_{rj}(t)$ for $j=1,2$, as defined in Eq.~\eqref{ur2pm}, under Assumptions~\ref{asp: global} and~\ref{asp: N}. 

Also see the work by Enss  \cite{enss1978asymptotic} in the context Quantum Scattering Theory. Many different ways of constructing incoming/outgoing decompositions were developed over the years, see\cite{mourre1979link, T2006book,soffer2011monotonic }.

\begin{definition}
    Let $S^{n-1}$ denote the standard unit sphere in $\mathbb{R}^n$. For $h\in \mathbb{R}^n\setminus\{0\}$, we define $\hat h:=h/|h|$ and $\hat h=0$ when $h=0$.

    We define a class of functions on $S^{n-1}$, $\{F^{\hat{h}}(\xi)\}_{\hat{h}\in I}$, as a smooth partition of unity with an index set 
 \eq
 I=\{ \hat{h}_1,\cdots, \hat{h}_N\}\subseteq S^{n-1}\label{indexI}
 \eeq
 for some $N\in \N^+$, satisfying that there exists $c>0$ such that for every $\hat{h}\in I$, 
 \eq
 F^{\hat{h}}(\xi)=\begin{cases}1 & \text{ when }|\xi-\hat{h}|<c\\ 0 & \text{ when }|\xi-\hat{h}|>2c\end{cases}, \quad \xi\in S^{n-1}.\label{ceq1}
 \eeq
 
Given $\hat{h}\in I$, we define $\tilde F^{\hat h}: S^{n-1}\to \R,$ as another smooth cut-off function satisfying 
 \eq
 \tilde{F}^{\hat{h}}(\xi)=\begin{cases}1 & \text{ when }|\xi-\hat{h}|<4c\\ 0 & \text{ when }|\xi-\hat{h}|>8c\end{cases},\quad \xi\in S^{n-1}.\label{ceq2}
 \eeq
We also assume that $c>0$, defined in \eqref{ceq1} and \eqref{ceq2}, is properly chosen  such that for all $x$ and $q$ within the supports of $F^{\hat{h}}$ and $\tilde F^{\hat{h}}$, respectively, 
\eq
F^{\hat{h}}(\hat{x})\tilde{F}^{\hat{h}}(\hat{q})|x+q|\geq \frac{1}{10}(|x|+|q|)F^{\hat{h}}(\hat{x})\tilde{F}^{\hat{h}}(\hat{q}),\label{Feq1}
\eeq
and
\eq
F^{\hat{h}}(\hat{x})(1-\tilde{F}^{\hat{h}}(\hat{q}))|x-q|\geq \frac{1}{10^6}(|x|+|q|)F^{\hat{h}}(\hat{x})(1-\tilde{F}^{\hat{h}}(\hat{q})).\label{Feq2}
\eeq
\end{definition}

Now let us define the projection on forward/backward propagation set with respect to the phase-space $(r,v)\in \R^{n+n}$:

\begin{definition}[Projection on the forward/backward propagation set] \label{def: Ppm}

We take 
\begin{equation}
    r=|x|,\quad v=\frac{p}{|p|+0},\quad P^\pm=P^\pm(|x|, \frac{p}{|p|+0}).
\end{equation}
Here, $0:=\lim\limits_{\epsilon\downarrow 0}\epsilon$. Then, the projections onto the forward and backward propagation sets, with respect to $(r,v)$, are defined as follows:
\eq
P^+(r,v):=\sum\limits_{b=1}^N F^{\hat{h}_b}(\hat{r})\tilde{F}^{\hat{h}_b}(\hat{v}),\label{Prv+}
\eeq
and 
\eq
P^-(r,v):=1-P^+(r,v),\label{Prv-}
\eeq
respectively. 
\end{definition}

We have the following lemma.

\begin{lemma}\label{lem: free: est1} Let $M\geq 0$ be a non-negative number. For all $n\geq 3$, $\sigma \in (1,n/2)$ and $ s\geq 0$ we have
\begin{equation}
   \| \chi(|x|\geq M) P^\pm e^{\pm i s|p|}\langle x\rangle^{-\sigma}\|\lesssim_{n/2-\sigma,\sigma,n} \frac{1}{\langle M+s\rangle^{\sigma}}.\label{est: Qmpm}
\end{equation}

\end{lemma}
\begin{proof} When $M+s\leq 1$, we obtain estimate~\eqref{est: Qmpm} by using the unitarity of $e^{\pm is|p|}$ and estimate $\|P^\pm\|\leq 1$. Therefore, in what follows we assume that~$M+s>1$. 

Let
\begin{equation}
    Q^\pm(s):=\chi(|x|\geq M) P^\pm e^{\pm i sH_0}\langle x\rangle^{-\sigma},
\end{equation}
where for~$P^\pm$ see~\eqref{Prv+},~\eqref{Prv-}. We estimate $Q^+(s)$ and note that $Q^-(s)$ is treated similarly. We take an arbitrary auxiliary function $f\in L^2(\mathbb{R}^n)$.

We decompose $Q^+(s)f$ into two parts:
\eq
Q^+(s)f=Q_{1}^+(s)f+Q_{2}^+(s)f,\label{def: Qm1-2}
\eeq
where 
\eq
Q_{1}^+(s)f:=\chi(|x|\geq M) P^+ e^{is|p|} F_1(|p|\leq \frac{1}{M+s})\langle x\rangle^{-\sigma}
\eeq
and
\eq
Q_{2}^+(s)f:=\chi(|x|\geq M) P^+ e^{is|p|} F_1(|p|> \frac{1}{M+s})\langle x\rangle^{-\sigma}.
\eeq
In view of the following estimates
\begin{equation}
    \begin{aligned}
        \| |p|^{-\sigma}|x|^{-\sigma}\| &   \lesssim_{n/2-\sigma} 1\qquad \forall \sigma \in (0,n/2),\\
        \| P^+e^{is|p|}\|   &   \leq 1,\\
        \| F_1(|p|\leq \frac{1}{M+s})|p|^\sigma\|   &   \lesssim \frac{1}{(M+s)^\sigma},
    \end{aligned}
\end{equation}
we obtain that $Q_{1}^+(s)f$ satisfies the following 
\begin{align}
    \| Q_{1}^+(s)f\|    & \leq \| \chi(|x|\geq M) P^+ e^{is|p|}\|\|F_1(|p|\leq \frac{1}{M+s})|p|^\sigma\|\| |p|^{-\sigma}|x|^{-\sigma}\|\| |x|^\sigma \langle x\rangle^{-\sigma}f\|\nonumber\\
    &   \lesssim_{n/2-\sigma}  \frac{1}{(M+s)^\sigma}\|f\|.\label{est: Qm1}
\end{align}
We decompose $Q_{2}^+(s)f$ further: 
\begin{equation}
    Q_{2}^+(s)f=Q_{21}^+(s)f+Q_{22}^+(s)f,\label{5.16}
\end{equation}
where for some $C_2>0$ which is given in Eq.~\eqref{def: C2},
\begin{equation}
    Q_{21}^+(s)f:=\chi(|x|\geq M) P^+e^{is|p|} F_1(|p|>\frac{1}{M+s})F_c(\frac{|x|}{C_2(M+s)}> 1) \langle x\rangle^{-\sigma }f
\end{equation}
and
\begin{equation}
    Q_{22}^+(s)f:=\chi(|x|\geq M) P^+e^{is|p|} F_1(|p|>\frac{1}{M+s})F_c(\frac{|x|}{C_2(M+s)}\leq  1) \langle x\rangle^{-\sigma }f.
\end{equation}

In view of the following estimates
\begin{equation}
    \begin{aligned}
        \| \chi(|x|\geq M) P^+e^{is|p|} F_1(|p|>\frac{1}{M+s})\|    &   \leq 1,\\
        \| F_c(\frac{|x|}{C_2(M+s)}> 1) \langle x\rangle^{-\sigma }\|   &   \lesssim \frac{1}{(M+s)^\sigma},
    \end{aligned}
\end{equation}
we obtain 
\begin{align}
    \|Q_{21}^+(s)f\|\leq &\| \chi(|x|\geq M) P^+e^{is|p|} F_1(|p|>\frac{1}{M+s})\|\| F_c(\frac{|x|}{C_2(M+s)}> 1) \langle x\rangle^{-\sigma }\|\|f\|\nonumber\\
    \lesssim_n & \frac{1}{(M+s)^\sigma}\|f\|.\label{est: Qm21}
\end{align}
Next, we estimate $Q_{22}^+(s)f$. By using the Fourier transform, we obtain that $Q_{22}^+(s)f$ reads as follows
\begin{align}
    Q_{22}^+(s)f=& \frac{1}{(2\pi)^{n/2}} \chi(|x|\geq M)\int e^{i(x\cdot q\pm s|q|)}P^\pm(|x|, \frac{q}{|q|+0}) e^{-0|q|}F_1(|q|>\frac{1}{M+s})\nonumber\\
    &\times\mathscr{F}[\langle x\rangle^{-\sigma}F_cf](q)dq,\label{F: def: Af}
\end{align}
with
\begin{equation}
    F_c\equiv F_c(\frac{|x|}{C_2(M+s)}\leq  1)
\end{equation}
and moreover~$\mathscr{F}[f(x)](q)$ denotes the Fourier transform of $f(x)$. Plugging 
\begin{equation}
    \mathscr{F}[\langle x\rangle^{-\sigma}F_cf](q)=\frac{1}{(2\pi)^{n/2}}\int e^{-iq\cdot y} \langle y\rangle^{-\sigma}F_cf(y)dy
\end{equation}
into Eq.~\eqref{F: def: Af}, we obtain 
\begin{align}
    Q_{22}^+(s)f=& \frac{1}{(2\pi)^n} \chi(|x|\geq M)\int e^{i(x\cdot q\pm s|q|)}P^\pm(|x|, \frac{q}{|q|+0}) e^{-0|q|}F_1(|q|>\frac{1}{M+s})\nonumber\\
    &\qquad\qquad\qquad\qquad\qquad\times e^{-iy\cdot q} \langle y\rangle^{-\sigma} F_cf(y)dydq.\label{F: def: Af0}
\end{align}
We define $\{e_1,\cdots,e_n\}$ as an orthogonal normal basis in $\mathbb{R}^n$ with $e_1$, being some direction such that 
\begin{equation}
    |x_1+ sq_1/|q|| \geq \frac{1}{n}|x+s\hat{q}|,  
\end{equation}
where $x_j:=x\cdot e_j$ and $q_j:=q\cdot e_j$ for $j=1,\cdots, n$. By Eqs.~\eqref{Feq1} and~\eqref{Prv+}, this implies that for all $x$ and $q$ in the support of $P^+(x,\frac{q}{|q|+0})$, 
\begin{equation}
    |x_1+sq_1/|q||\geq \frac{1}{10n}(|x|+s)\geq \frac{1}{10^6n}(|x|+s).
\end{equation}

 We take 
\eq
C_2:=\frac{1}{10^{10}n},\label{def: C2}
\eeq
and use the cut off~\eqref{def: F}, to define $F_1$, where note that the latter implies $|y|\leq C_2(M+s)$. We obtain that for $|x|\geq M$ the following holds  
\begin{align}
     |x_1+sq_1/|q|-y_1|\geq \frac{1}{10^6n}(|x|+s)-\frac{1}{10^{10}n}(M+s)\geq \frac{1}{10^7n} (|x|+s).\label{station: est}
\end{align}  

We integrate by parts the right-hand side of Eq.~\eqref{F: def: Af0} which we use in conjunction with
\begin{equation}
    e^{i(x_1q_1+ s|q|-y_1q_1)-0|q|}=\frac{1}{i(x_1+ sq_1/|q|-y_1)-0q_1/|q|}\p_{q_1}[ e^{i(x_1q_1+ s|q|-y_1q_1)-0|q|} ],
\end{equation}
to obtain 
\begin{align}
    Q_{22}^+(s)f=&\mathcal I_{1}^+(s) f+\mathcal I_{2}^+(s)f,\label{integrate: once: id}
\end{align}
where
\begin{align}
\mathcal I_{1}^+(s)f:=&\frac{-\chi(|x|\geq M)}{(2\pi)^n} \int \frac{e^{i(x\cdot q+ s|q|-y\cdot q)-0|q|}\partial_{q_1}[P^+(|x|, \frac{q}{|q|+0}) F_{1}]e^{-iy\cdot q} \langle y\rangle^{-\sigma}F_cf(y)}{i(x_1+sq_1/|q|-y_1)-0q_1/|q|} dydq,\label{integrate: once: def1}
\end{align}
and
\begin{align}
\mathcal I_{2}^+(s)f:=&\frac{-\chi(|x|\geq M)}{(2\pi)^n} \int \p_{q_1}[\frac{1}{i(x_1+sq_1/|q|-y_1)-0q_1/|q|}]\nonumber\\
&\times e^{i(x\cdot q+ s|q|-y\cdot q)-0|q|}P^+(|x|, \frac{q}{|q|+0})F_{1} e^{-iy\cdot q} \langle y\rangle^{-\sigma}F_cf(y)dydq.\label{integrate: once: def2}
\end{align}

A straightforward computation shows that 
\begin{equation}
    |\p_{q_1}[ P^+(|x|,\frac{q}{|q|+0})]|=|\p_2[ P^+(|x|,\frac{q}{|q|+0})] \frac{q_1}{(|q|+0)^2}|\lesssim \frac{1}{|q|+0}, \label{integrate: once: eq1}
\end{equation}
where~$\p_2 [P^+(|x|,r)]:=\p_r[P^+(|x|,r)]$ and moreover 
\begin{align}
    |\p_{q_1}[F_1(|q|>\frac{1}{M+s})]|=&|(M+s)F_1'(|q|>\frac{1}{M+s}) \frac{q_1}{|q|}|\nonumber\\
    =& |(M+s)|q|F_1'(|q|>\frac{1}{M+s}) \frac{q_1}{|q|}|\frac{1}{|q|}\lesssim \frac{1}{|q|},
\end{align}
where~$F_1'\equiv \frac{d}{dk}[F_1(k)]$. 

Now, we use the estimate~\eqref{station: est} to obtain
\begin{equation}
    |\frac{1}{i(x_1+sq_1/|q|-y_1)-0q_1/|q|}|\lesssim_n \frac{1}{|x|+s}\label{integrate: once: eq2}
\end{equation}
and
\begin{align}
    | \p_{q_1}[\frac{1}{i(x_1+sq_1/|q|-y_1)-0q_1/|q|}]|= &|\frac{is-0}{\left(i(x_1+sq_1/|q|-y_1)-0q_1/|q|\right)^2} \frac{q_1}{(|q|+0)^2}|\nonumber\\
    \lesssim_n& \frac{1}{(|x|+s)(|q|+0)}.\label{integrate: once: eq3}
\end{align}

The estimates~\eqref{integrate: once: eq1},~\eqref{integrate: once: eq2} and~\eqref{integrate: once: eq3}, imply that by integrating by parts~\eqref{integrate: once: def1} and~\eqref{integrate: once: def2} we obtain
\begin{equation}
   \frac{1}{(|x|+s)|q|} 
\end{equation}
decay. Hence, by using estimate~\eqref{station: est}, we integrate by parts $n$ more times for both $\mathcal I_1^+(s)f$ and $\mathcal I_2^+(s)f$ in the same way to obtain the following pointwise estimates 
\begin{align}
    |[\mathcal I_j^+(s)f](x)|\lesssim_n& \frac{1}{(|x|+s)^{n+1}} \int \frac{1}{|q|^{n+1}}(F_1( |q|>\frac{1}{2(M+s)})F_1(|q|\leq  2)  \langle y\rangle^{-\sigma}F_c|f(y)| dydq,\qquad j=1,2 \label{Ij: pt-est: eq1}
\end{align}
for $|x|\geq M$. Integrating over $y$ and $q$ in Eqs.~\eqref{Ij: pt-est: eq1}, we obtain
\begin{align}
    |[\mathcal I_j^+(s)f](x)|\lesssim_n& \frac{M+s}{(|x|+s)^{n+1}}\|\langle y\rangle^{-\sigma}F_cf(y)\|_{L^1_y(\mathbb{R}^n)}\nonumber\\
    \leq & \frac{1}{(|x|+s)^n}\|\langle y\rangle^{-\sigma}F_cf(y)\|_{L^1_y(\mathbb{R}^n)},\qquad j=1,2\label{Ij: pt-est: eq2}
\end{align}
for $|x|\geq M$.\par 

By using H\"older's inequality, estimates~\eqref{Ij: pt-est: eq2} in conjunction with
\begin{equation}
    \| \langle y\rangle^{-\sigma}F_c(\frac{|y|}{C_2(M+s)}\leq 1)\|\lesssim (M+s)^{n/2-\sigma},
\end{equation}
we obtain
\begin{align}
     |[\mathcal I_j^+(s)f](x)|\lesssim_n& \frac{1}{(|x|+s)^{n}} \|\langle y\rangle^{-\sigma}F_c\|\|f\|\nonumber\\
     \lesssim_n & \frac{(M+s)^{n/2-\sigma}}{(|x|+s)^{n}}\|f\|\nonumber\\
     \lesssim_n & \frac{1}{(|x|+s)^{n/2+\sigma}}\|f\|
\end{align}
for $j=1,2$ and $|x|\geq M$. This implies 
\begin{equation}
    \| \chi(|x|\geq M) \mathcal I_j^+(s) f\|\lesssim_{n,\sigma} \frac{1}{(M+s)^\sigma}\|f\|,\qquad j=1,2.\label{est: Ij} 
\end{equation}
Estimate~\eqref{est: Ij} and Eq.~\eqref{integrate: once: id} imply 
\begin{equation}
    \| Q_{22}^+(s)f\|\lesssim_{n,\sigma} \frac{1}{(M+s)^{\sigma}}\|f\|.\label{est: Qm22}
\end{equation}
Estimates~\eqref{est: Qm1},~\eqref{est: Qm21} and~\eqref{est: Qm22}, together with Eqs.~\eqref{def: Qm1-2} and~\eqref{5.16}, imply 
\begin{equation}
    \| Q^+(s)f\|\lesssim_{n/2-\sigma,n,\sigma} \frac{1}{(M+s)^\sigma}\|f\|,
\end{equation}
which yields estimate~\eqref{est: Qmpm} for $Q^+(s)$. Similarly, we obtain estimate~\eqref{est: Qmpm} for $Q^-(s)$.
\end{proof}

\begin{lemma}\label{lem: free: est2} Let~$n\geq 3,~\sigma\in (1,n/2)$,~$s\geq 0$ and~$\epsilon\in (0,\sigma-1)$. Then, we have
\begin{equation}
    \| \langle x\rangle^\epsilon P^\pm e^{\pm is|p|}\langle x\rangle^{-\sigma}\|\lesssim_{n/2-\sigma, \sigma ,n}\frac{\log_2(s+2)}{(1+s)^{\sigma-\epsilon}}\in L^1_s[1,\infty),
\end{equation}
where for~$P^\pm$ see Eqs.~\eqref{Prv+} and~\eqref{Prv-}. 
\end{lemma}
\begin{proof} Let 
\eq
Q^\pm(s):=\langle x\rangle^\epsilon P^\pm e^{\pm is|p|}\langle x\rangle^{-\sigma}.
\eeq
We estimate $Q^+(s)$ and note that $Q^-(s)$ is treated similarly. Take $f\in L^2_x(\mathbb{R}^n)$. We decompose $Q^+(s)f$ into several parts:
\eq
Q^+(s)f=\chi(|x|<1)Q^+(s)f+\sum\limits_{j=0}^\infty Q^+_j(s)f,
\eeq
where 
\eq
Q^+_j(s)f:=\langle x\rangle^\epsilon \chi(|x|\in [2^j,2^{j+1}))P^+e^{is|p|}\langle x\rangle^{-\sigma},\qquad\qquad j=0,1,\cdots.
\eeq
By Lemma~\ref{lem: free: est1} with $M=2^j$, we find 
\begin{equation}
     \| \langle x\rangle^\epsilon\chi(|x|<1) Q^+(s)f\|\leq\|Q^+(s)f\|\lesssim_{n/2-\sigma,n/2,\sigma} \|f\|
\end{equation}
and
\begin{align}
    \| Q_j^+(s)f\|\leq & \| \langle x\rangle^\epsilon \chi(|x|\in [2^j,2^{j+1}))\|\| \chi(|x|\geq 2^j) P^+e^{is|p|}\langle x\rangle^{-\sigma}\|\|f\|\lesssim_{n/2-\sigma,n/2,\sigma}\frac{\|f\|}{(2^j+s)^{\sigma-\epsilon}}.\label{est: Qj+}
\end{align}
We sum estimate~\eqref{est: Qj+} over~$j\in \mathbb{N}$ to obtain 
\begin{align}
   \| Q^+(s)f\| & \leq \| \chi(|x|<1)Q^+(s)f\|+\sum\limits_{j=0}^\infty \| Q_j^+(s)f\|\lesssim_{n/2-\sigma,n/2,\sigma} \|f\|+\sum\limits_{j=0}^\infty\frac{\|f\|}{(2^j+s)^{\sigma-\epsilon}}\nonumber\\
   &    \lesssim_{n/2-\sigma,n/2,\sigma} \frac{\log_2(s+2)}{(1+s)^{\sigma-\epsilon}}\|f\|.
\end{align}
Similarly, we obtain 
\begin{align}
   \| Q^-(s)f\|\lesssim_{n/2-\sigma,n/2,\sigma} &\frac{\log_2(s+2)}{(1+s)^{\sigma-\epsilon}}\|f\|.
\end{align}
    
\end{proof}
We formally define the following integrals
\begin{equation}
    u_{r1}^\pm(t):=\int_t^\infty P^\pm e^{\mp i(t-s)|p|}N(u(s))u(s)ds,\qquad  u_{r2}^\pm(t):=\int_0^t P^\pm e^{\pm i(t-s)|p|}N(u(s))u(s)ds,\label{ur2pm}
\end{equation}
with $P^\pm$ defined in Definition~\eqref{def: Ppm} and~$u$ here is a smooth function.

\begin{proposition}\label{prop: urj} Let~$u$ be a solution of the nonlinear wave~\eqref{eq: wave} and let Assumptions~\ref{asp: global} and~\ref{asp: N} hold. Moreover, let~$\delta \in (0,\sigma-1)$. Then, we have
\begin{equation}
    \| \langle x\rangle^\delta u_{rj}^\pm(t)\|\lesssim_{n/2-\sigma,\sigma,n,\sigma-1-\delta} \sup\limits_{s\geq 0}\|\langle x\rangle^\sigma N(u(s))u(s)\|,\qquad j=1,2,
\end{equation}
where for~$u^\pm_{rj}$ see~\eqref{ur2pm}.
\end{proposition}

\begin{proof}

The present proof is a direct consequence of Lemma~\ref{lem: free: est2}. Namely from the estimates of Lemma~\ref{lem: free: est2} we obtain 
\begin{align}
     \| \langle x\rangle^\delta u_{r1}^\pm(t)\|\lesssim_{n/2-\sigma,\sigma,n}& \int_t^\infty \frac{1}{(1+s)^{\sigma-\delta}}\sup\limits_{s\geq 0}\|\langle x\rangle^\sigma N(x,s,u(s))u(s)\|ds\nonumber\\
     \lesssim_{n/2-\sigma,\sigma,n,\sigma-1-\delta}&\sup\limits_{s\geq 0}\|\langle x\rangle^\sigma N(x,s,u(s))u(s)\|
\end{align}
and
\begin{align}
     \| \langle x\rangle^\delta u_{r2}^\pm(t)\|\lesssim_{n/2-\sigma,\sigma,n}& \int_0^t \frac{1}{(1+s)^{\sigma-\delta}}\sup\limits_{s\geq 0}\|\langle x\rangle^\sigma N(x,s,u(s))u(s)\|ds\nonumber\\
     \lesssim_{n/2-\sigma,\sigma,n,\sigma-1-\delta}&\sup\limits_{s\geq 0}\|\langle x\rangle^\sigma N(x,s,u(s))u(s)\|,
\end{align}
which readily concludes the proof. 
\end{proof}

\section{Proof of Theorem~\ref{main theorem 1}}\label{sec: proof of thm1}

\subsection{Proof of the first part of Theorem~\ref{main theorem 1}}

We establish the existence of the free channel wave operator 
\eq
\Omega_{\alpha}^*\vu(0) := s\text{-}\lim_{t\to \infty} \mathcal{F}_\alpha(x,p,t) U_0(0,t) \vu(t),
\eeq
which was formally defined in Definition~\ref{def: subsecL free channel, def 1}. This proof utilizes the propagation estimates introduced in Section~\ref{subsec: sec: WO, subsec 3}, and adheres to Assumptions \ref{asp: global} and \ref{asp: N}.

Consider the vector
$$
\vu_{\Omega,\alpha}(t):=\mathcal{F}_\alpha(x,p,t)\Omega(t)^*\vu(0).
$$
We expand $\vu_{\Omega,\alpha}(t)$ using Cook's method as follows
\eq
\begin{split}
\vu_{\Omega,\alpha}(t)=&\vu_{\Omega,\alpha}(1)-\int_1^tds \mathcal{F}_\alpha(x,p,t)U_0(0,s)\mathcal{N}(\vec{u}(s))\vu(s)\\
&+\int_1^tds \p_s[\mathcal{F}_\alpha(x,p,s)]U_0(0,s)\vu(s)\\
=:&\vu_{\Omega,\alpha}(1)+\vu_{in}(t)+\vu_p(t).
\end{split}
\eeq
In view of Assumption~\ref{asp: global} and moreover since $\sup\limits_{t\in \mathbb{R}}\|U_0(t,0)\|_{\Hi\to \Hi}<\infty$
we obtain 
\begin{equation}
    \vu_{\Omega,\alpha}(1)\in \Hi.
\end{equation}

Now, recall that
\eq
\mathcal{F}_\alpha(x,p,t)=\begin{pmatrix}
 |p|^{-1}F_\alpha(x,p,t)|p| &0 \\
 0 &  F_\alpha(x,p,t)
 \end{pmatrix},\quad \alpha\in (0, 1-1/\sigma),
 \eeq
 where $F_1$ is defined in~\eqref{def: F1pt} and $F_\alpha$ is given by  
 \eq
     F_\alpha(x,p,t):=F_c(\frac{|x|}{t^\alpha}\leq 1)F_1(|p|,t).
 \eeq

Therefore, applying the dispersive estimate~\eqref{decay: free} of Lemma~\ref{lem: decay}, we find that for all $\alpha\in (0, 1-1/\sigma)$ the following inequality holds
\eq
\begin{split}
\| \vu_{in}(t)\|_{\Hi}\leq &\int_{1}^{t} \| \mathcal{F}_\alpha(x,p,t)U_0(0,s) \mathcal{N}(x,s,\vu(s))\vu(s)\|_{\Hi}ds\\
\lesssim &\int_{1}^{t} \frac{1}{s^{(1-\alpha)\sigma}} \|\langle x\rangle^\sigma N(x,s,u(s))u(s)\|ds\\
\lesssim_{(1-\alpha)\sigma-1} &\sup\limits_{t\in \R} \|\langle x\rangle^\sigma N(x,t,u)u(t)\|.
\end{split}
\eeq
Consequently, $\vu_{in}(\infty)$ exists in $\Hi$. For $\vu_p(t)$, we use the relative propagation estimates with $\vec{v}=U_0(0,t)\vec{u}(t)$ and the propagation observables $\{B_1(t)\}_{t\geq 0}, \{B_2(t)\}_{t\geq 0}$, given in \eqref{B1}, \eqref{B2}, respectively. 

We compute 
\begin{align}
    \vu_p(t)= & \int_1^t ds\mathcal{F}_c(x,p,s)\vec{v}(s)+  \int_1^t ds\mathcal{F}_1(x,p,s)\vec{v}(s)+ \int_1^t ds\mathcal{F}_r(x,p,s)\vec{v}(s)\nonumber\\
    =& \vu_{pc}(t)+\vu_{p1}(t)+\vu_{pr}(t)
\end{align}
where $\mathcal{F}_c(x,p,s), \mathcal{F}_1(x,p,s) $ and $\mathcal{F}_r(x,p,s)$ are given by 
\begin{equation}
    \begin{aligned}
        \mathcal{F}_c(x,p,s)    &   = \begin{pmatrix}
    |p|^{-1} \partial_t[F_c]F_1 |p| & 0\\
    0& \partial_t[F_c]F_1
\end{pmatrix},\\
\mathcal{F}_1(x,p,s)    &   = \begin{pmatrix}
    |p|^{-1} \partial_t[F_1]F_c |p| & 0\\
    0& \partial_t[F_1]F_c
\end{pmatrix},\\
\mathcal{F}_r(x,p,s)    &   = \begin{pmatrix}
    |p|^{-1} [F_c, \partial_t[F_1]] |p| & 0\\
    0&[F_c,  \partial_t[F_1]]
\end{pmatrix},
    \end{aligned}
\end{equation}
respectively. Now, in view of Lemma \ref{com} we conclude that 
\begin{equation}
    \|\mathcal{F}_r(x,p,s)\vec{v}(s)\|_{\Hi}\in L^1_s[1,\infty).
\end{equation}
Therefore, $\vu_{pr}(\infty) $ exists in $\Hi$. For $\vu_{pc}(t)=(u_{pc1}(t), u_{pc2}(t))$, by employing H\"older's inequality in $s$ variable, we find that for all $t_1>t_2\geq T\geq 1$,
\begin{align}
||p|u_{pc1}(t_1) -|p|u_{pc1}(t_2) |\leq & \int_{t_2}^{t_1} ds|\p_t[F_c]\vert_{t=s}||F_1|p|v_1(s) |  \nonumber\\
\leq & \left( \int_{t_2}^{t_1} ds \p_t[F_c]\vert_{t=s}\right)^{1/2}\times\left( \int_{t_2}^{t_1} ds \p_t[F_c]\vert_{t=s} |F_1|p|v_1(s) |^2\right)^{1/2}\nonumber\\
\leq & \left( \int_{t_2}^{t_1} ds \p_t[F_c]\vert_{t=s} |F_1|p|v_1(s) |^2\right)^{1/2}
\end{align}
and 
\begin{align}
    |u_{pc2}(t_1)-u_{pc2}(t_2)|\leq & \int_{t_2}^{t_1} ds|\p_t[F_c]\vert_{t=s}||F_1v_2(s) |  \nonumber\\
\leq & \left( \int_{t_2}^{t_1} ds \p_t[F_c]\vert_{t=s}\right)^{1/2}\times\left( \int_{t_2}^{t_1} ds \p_t[F_c]\vert_{t=s} |F_1v_2(s) |^2\right)^{1/2}\nonumber\\
\leq & \left( \int_{t_2}^{t_1} ds \p_t[F_c]\vert_{t=s} |F_1v_2(s) |^2\right)^{1/2},
\end{align}
where we have used the fact that $\p_t[F_c]\geq 0$. Therefore, by Lemmata~\ref{lem: prop: B1B2} and~\ref{lem: propest: B1B2} we find that as $T\to \infty,$
\begin{align}
    \| |p|u_{pc1}(t_1)-|p|u_{pc1}(t_2)\|_{L^2_x(\mathbb{R}^n)}\leq & \left( \int_{t_2}^{t_1} ds (F_1|p|v_1(s),\p_t[F_c]\vert_{t=s} F_1|p|v_1(s))_{L^2_x(\mathbb{R}^n)} \right)^{1/2}\nonumber\\
    \to & 0 
\end{align}
and 
\begin{align}
    \| u_{pc2}(t_1)-u_{pc2}(t_2)\|_{L^2_x(\mathbb{R}^n)}\leq & \left( \int_{t_2}^{t_1} ds (F_1v_2(s),\p_t[F_c]\vert_{t=s} F_1v_2(s))_{L^2_x(\mathbb{R}^n)} \right)^{1/2}\nonumber\\
    \to & 0.
\end{align}
Consequently, $ \{\vu_{pc}(t)\}_{t\geq 1}$ is Cauchy in $\Hi$ and thus, $\vec{u}_{pc}(\infty)$ exists in $\Hi$. 

Let 
\begin{equation}
    \hat{\vu}:=\mathscr{F}_x[\vu](k)
\end{equation}
denote the Fourier transform in $x$ variable. Similarly, for $\vu_{p1}(t)=(u_{p11}(t), u_{p12}(t))$, by employing H\"older's inequality in $s$ variable in the Fourier space of $x$, we obtain that for all $t_1>t_2\geq T\geq 1$ the following hold
\begin{align}
| |k| \hat{u}_{p11}(t_1) -&|k|\hat{u}_{p11}(t_2) |\leq  \int_{t_2}^{t_1} ds|\p_t[\hat{F}_1(|k|,t)]\vert_{t=s}|\mathscr{F}_x[F_c |p|v_1(s)](k) |  \nonumber\\
\leq & \left( \int_{t_2}^{t_1} ds \p_t[\hat{F}_1(|k|,t)]\vert_{t=s}\right)^{1/2}\times\left( \int_{t_2}^{t_1} ds  \p_t[\hat F_1(|k|,t)]\vert_{t=s} |\mathscr{F}_x[F_c|p|v_1(s) ](k)|^2\right)^{1/2}\nonumber\\
\leq & \left( \int_{t_2}^{t_1} ds  \p_t[F_1(|k|,t)]\vert_{t=s} |\mathscr{F}_x[F_c|p|v_1(s) ](k)|^2\right)^{1/2}
\end{align}
and 
\begin{align}
   |  \hat{u}_{p12}(t_1) -&\hat{u}_{p12}(t_2) |\leq  \int_{t_2}^{t_1} ds|\p_t[\hat{F}_1(|k|,t)]\vert_{t=s}|\mathscr{F}_x[F_c v_2(s)](k) |  \nonumber\\
\leq & \left( \int_{t_2}^{t_1} ds \p_t[\hat{F}_1(|k|,t)]\vert_{t=s}\right)^{1/2}\times\left( \int_{t_2}^{t_1} ds  \p_t[\hat F_1(|k|,t)]\vert_{t=s} |\mathscr{F}_x[F_cv_2(s) ](k)|^2\right)^{1/2}\nonumber\\
\leq & \left( \int_{t_2}^{t_1} ds  \p_t[F_1(|k|,t)]\vert_{t=s} |\mathscr{F}_x[F_cv_2(s) ](k)|^2\right)^{1/2},
\end{align}
where we have used the fact that $\p_t[\hat{F}_1(|k|,t)]\geq 0$. Therefore, by Lemmata~\ref{lem: prop: B1B2} and~\ref{lem: propest: B1B2} and Plancherel theorem, we find that as $T\to \infty,$
\begin{align}
    \| |p|u_{p11}(t_1)-&|p|u_{p11}(t_2)\|_{L^2_x(\mathbb{R}^n)}=\| |k|\hat{u}_{p11}(t_1)-|k|\hat{u}_{p11}(t_2)\|_{L^2_k(\mathbb{R}^n)}\nonumber\\
    \leq &\left( \int_{t_2}^{t_1} ds  (\mathscr{F}_x[F_cv_1(s) ](k),\p_t[F_1(|k|,t)]\vert_{t=s} \mathscr{F}_x[F_cv_1(s) ](k))_{L^2_k(\mathbb{R}^n)}\right)^{1/2}\nonumber\\
    \to & 0 
\end{align}
and 
\begin{align}
    \| u_{p12}(t_1)-&u_{p12}(t_2)\|_{L^2_x(\mathbb{R}^n)}=\| \hat{u}_{p12}(t_1)-\hat{u}_{p12}(t_2)\|_{L^2_k(\mathbb{R}^n)}\nonumber\\
    \leq & \left( \int_{t_2}^{t_1} ds  (\mathscr{F}_x[F_cv_2(s) ](k),\p_t[F_1(|k|,t)]\vert_{t=s} \mathscr{F}_x[F_cv_2(s) ](k))_{L^2_k(\mathbb{R}^n)}\right)^{1/2}\nonumber\\
    \to & 0 .
\end{align}
Hence, $ \{\vu_{p1}(t)\}_{t\geq 1}$ is Cauchy in $\Hi$ and therefore, $\vec{u}_{p1}(\infty)$ exists in $\Hi$. 

In view of Lemma \ref{com} we obtain  
\begin{equation}
    \|\mathcal{F}_r(x,p,s)\vec{v}(s)\|_{\Hi}\in L^1_s[1,\infty).
\end{equation}
Consequently, $\vu_{pr}(\infty)$ exists in $\Hi$. Therefore, all elements including $\vu_{pc}(\infty), \vu_{p1}(\infty)$ and $\vu_{pr}(\infty)$ exist in $\Hi$, which leads to the existence of $\vu_p(\infty)$ in $\Hi$. Thus, both $\vu_{in}(\infty)$ and $\vu_{p}(\infty)$ exist in $\Hi$ and as a result, $\vu_{\Omega,\alpha}(\infty)$ also exists in $\Hi$. 

Next, we prove~\eqref{local: weak: eq}. We use that for all $\alpha\in (0,1-1/\sigma)$ the following holds
\begin{equation}
    w\text{-}\lim\limits_{t\to\infty} (1-F_{\alpha}(x,p,t))=0,\qquad \text{ on }L^2_x(\mathbb{R}^n),
\end{equation}
to obtain 
\begin{equation}
    w\text{-}\lim\limits_{t\to\infty} (F_{\alpha'}(x,p,t)-F_{\alpha}(x,p,t))=0,\qquad \text{ on }L^2_x(\mathbb{R}^n)
\end{equation}
for all $\alpha,\alpha'\in (0,1-1/\sigma)$. Therefore, in view of~\eqref{Falpha: matrix} we obtain
\begin{equation}
      w\text{-}\lim\limits_{t\to\infty} (\mathcal F_{\alpha'}(x,p,t)-\mathcal F_{\alpha}(x,p,t))=0,\qquad \text{ on }\Hi,
\end{equation}
which, in turn, yields equation~\eqref{local: weak: eq} for all $\alpha,\alpha'\in (0, 1-1/\sigma)$.

\subsection{A decomposition of~$\Omega^*\vu (0)$}

In what follows, we decompose $\Omega^*\vu(0)$ in terms of $e^{\pm i t|p|}$ flows and it will be used in the proof regarding the properties of the non-free part, in Section~\ref{sec: local properties}. We use Lemma~\ref{rep: ut} to obtain
\begin{align}
\Omega^*\vu(0)=&w\text{-}\lim\limits_{t\to \infty} U_0(0,t)\vu(t)\nonumber\\
=& \vu(0)+\int_0^\infty U_0(0,s)\mathcal N(x,s,\vu(s))\vu(s)ds.
\end{align}

We denote $\Omega^*\vu(0)=(u_{\Omega,1}(x),u_{\Omega,2}(x))$ and $\vu(0)=(u(x),\dot u(x))$, which in conjunction with~\eqref{d_omegau},~\eqref{omegau} we obtain 
\begin{equation}
   |p| u_{\Omega,1}(x)=|p| u(x)+\int_0^\infty \sin(s|p|) N(x,s,u(s))u(s)ds
\end{equation}
and
\begin{equation}
    u_{\Omega,2}(x)=\dot u(x)+\int_0^\infty \cos(s|p|) N(x,s,u(s))u(s)ds.
\end{equation}

Now, by recalling the identities
\begin{equation}
    \begin{aligned}
        \sin(s|p|)=\frac{1}{2i}(e^{is|p|}-e^{-is|p|}),\qquad\cos(s|p|)=\frac{1}{2}(e^{is|p|}+e^{-is|p|}),
    \end{aligned}
\end{equation}
we obtain
\begin{equation}
    |p|u_{\Omega,1}(x)=|p|u(x)-\frac{1}{2i}u^-(x)+\frac{1}{2i}u^+(x)\label{omega: Peq1}
\end{equation}
and
\begin{equation}
    u_{\Omega,2}(x)=\dot u(x)+\frac{1}{2}u^-(x)+\frac{1}{2}u^+(x),\label{omega: noPeq1}
\end{equation}
where $u^\pm(x), j=1,2,$ are given by 
\eq
u^+=\int_0^\infty e^{is|p|} N(x,s,u(s))u(s)ds,\qquad u^-=\int_0^\infty e^{-is|p|}N(x,s,u(s))u(s)ds,
\eeq
respectively. 
\begin{lemma}\label{Lem: L2upm} Let~$u$ be a solution of~\eqref{eq: wave}. If Assumptions~\ref{asp: global} and~\ref{asp: N} are satisfied, then for $n\geq 3$ we have 
\begin{equation}
    u^\pm \in L^2_x(\mathbb{R}^n).
\end{equation}    
\end{lemma}
\begin{proof}Since $|p|u_{\Omega,1}(x), |p|u(x), u_{\Omega,2}, \dot u(x)\in L^2_x(\mathbb{R}^n)$, by Eqs~\eqref{omega: Peq1} and~\eqref{omega: noPeq1}, we obtain 
\eq
\frac{-1}{2i}u^-(x)+\frac{1}{2i}u^+(x)\in L^2_x(\mathbb{R}^n)
\eeq
and
\eq
\frac{1}{2}u^-(x)+\frac{1}{2}u^+(x)\in L^2_x(\mathbb{R}^n),
\eeq
which imply that $u^\pm(x)\in L^2_x(\mathbb{R}^n)$.

\end{proof}

\subsection{Proof of the second part of Theorem~\ref{main theorem 1}}\label{sec: local properties}

Denote $\vu(t)=(u(t),\dot u(t))$ and $\vu(0)=(u(x),\dot u(x))$. We use Lemma~\ref{rep: ut} to obtain  
\eq
u(t)=\cos(t|p|)u(x)+\frac{\sin(t|p|)}{|p|}\dot u(x)+\int_0^t \frac{\sin((t-s)|p|)}{|p|}N(u(s))u(s)ds 
\eeq
and
\eq
\dot u(t)=-\sin(t|p|)|p|u(x)+\cos(t|p|)\dot u(x)-\int_0^t \cos((t-s)|p|)N(u(s))u(s)ds,
\eeq
which implies 
\eq
|p|u(t)=\cos(t|p|)|p|u(x)+\sin(t|p|)\dot u(x)+\frac{1}{2i}e^{it|p|}u^-(t)-\frac{1}{2i}e^{-it|p|}u^+(t)\label{|p|u: eq}
\eeq
and
\eq
\dot u(t)=-\sin(t|p|)|p|u(x)+\cos(t|p|)\dot u(x)-\frac{1}{2}e^{it|p|}u^-(t)-\frac{1}{2}e^{-it|p|}u^+(t),\label{dotu: eq}
\eeq
where $u^\pm(t)$ are given by 
\eq
u^+(t)=\int_0^t e^{is|p|}N(x,s,u(s))u(s)ds
\eeq
and
\eq
u^-(t)=\int_0^t e^{-is|p|}N(x,s,u(s))u(s)ds,
\eeq
respectively.

Recall that $u^\pm_{rj}, j=1,2,$ are defined in Equation~\eqref{ur2pm}. We use Proposition~\ref{prop: urj} to obtain
\begin{align}
\| \langle x\rangle^\epsilon(P^\pm e^{ \mp it|p|}u^\pm(t)-P^\pm e^{\mp it|p|}u^\pm)\|=& \| \langle x\rangle^\epsilon u^\pm_{r1}(t)\|\nonumber\\
\lesssim_{n/2-\sigma,\sigma,n,\sigma-1-\epsilon} & \sup\limits_{s\geq 0}\| \langle x\rangle^\sigma N(x,s,u(s))u(s)\|
\end{align}
for all $\epsilon\in (0, \sigma-1)$, and moreover we obtain 
\begin{align}
\| \langle x\rangle^\epsilon P^\mp e^{ \mp it|p|}u^\pm(t)\|=& \| \langle x\rangle^\epsilon u^\mp_{r2}(t)\|\nonumber\\
\lesssim_{n/2-\sigma,\sigma,n,\sigma-1-\epsilon} & \sup\limits_{s\geq 0}\| \langle x\rangle^\sigma N(x,s,u(s))u(s)\|.
\end{align}

Additionally, we use~\eqref{ur2pm} to obtain
\eq
e^{ -it|p|}u^+(t)=P^+e^{ -it|p|}u^+(t)+u_{r2}^-(t)\label{decom: u+}
\eeq
and
\eq
e^{ it|p|}u^-(t)=P^-e^{ it|p|}u^-(t)+u_{r2}^+(t).\label{decom: u-}
\eeq

We note that~\eqref{decom: u+} and~\eqref{decom: u-} imply the following
\eq
\| e^{- it|p|}u^+(t)-e^{- i t|p|} u^+ -u_{r2}^-(t)\|=\| P^- e^{- i t|p|} u^+ \|\to 0\qquad \text{ as }t\to \infty
\eeq
and
\eq
\| e^{it|p|}u^-(t)-e^{ i t|p|} u^--u_{r2}^+(t)\|=\| P^+ e^{i t|p|} u^- \|\to 0\qquad \text{ as }t\to \infty.
\eeq
Again, in view of Proposition~\ref{prop: urj}, the above considerations in conjunction with equations~\eqref{|p|u: eq},~\eqref{dotu: eq}, imply the desired decomposition~\eqref{decom: goal} with
\eq
\vu_{lc}(t)=(|p|^{-1}\frac{1}{2i}(u_{r2}^+(t)-u_{r2}^-(t) ), -\frac{1}{2}( u_{r2}^+(t)+u_{r2}^-(t))  ).
\eeq
We conclude.

\section{Proof of Theorem~\ref{main theorem 3}} 

We verbatim follow the argument of the first part of Theorem~\ref{main theorem 1}, see Section~\ref{sec: proof of thm1}, by taking 
\begin{equation}
    F_1=F_1(|p|\leq t^\alpha)
\end{equation}
in the place of $F_1=F_1(|p|,t)$ and by using estimates~\eqref{est: decay non-local: goal} and~\eqref{est: decay non-local: goalD} instead of estimate~\eqref{decay: free}. Similarly, we obtain the existence of 
\begin{equation}
\Omega_\alpha^*\vu(0)
\end{equation}
in $\Hi$ and directly note its independence from $\alpha\in (0,\frac{n-3}{2n+1})$.

Let 
\eq
\Omega^*\vu(0):=w\text{-}\lim\limits_{t\to \infty} U_0(0,t)\vu(t),\qquad \text{ in }\Hi.
\eeq
The existence of $\Omega^*\vu(0)$ follows from Cook's method and estimate of free flow (see~\eqref{freeest: 1}). By Equation~\eqref{local: weak: eq}, we have 
\eq
\Omega^*\vu(0)=\Omega^*_\alpha\vu(0),\qquad \alpha\in (0,\frac{n-3}{2n+1}).
\eeq
We conclude.

\section{Proof of Theorem~\ref{thm: 3}}

By linearity, it suffices to show that the following quantities are in $L^1_t[1,\infty)$: 
\begin{equation}
   \|F_\alpha(x,p,t)e^{\pm i t|p|}\partial_i h^{ij}(x,t)\partial_j u(t)\|.  
\end{equation}
By the free dispersive estimate~\eqref{decay ineq} we obtain
\begin{equation*}
    \| \langle x\rangle^{-2}e^{\pm i t|p|}\partial_jF_1(|p|\leq t^\beta)\langle x\rangle^{-2}\| \lesssim_n \frac{1}{t^{2-\beta}} \qquad\forall \, \beta\in (0,1/3),\, t\geq 1,
\end{equation*}
for any index~$j$. Then we derive that
 \begin{align}
& \|F_\alpha(x,p,t)e^{\pm i t|p|}\partial_i h^{ij}(x,t)\partial_j u(t)\|\nonumber\\
\leq & \|F_1(|x|\leq t^\alpha)\langle x\rangle^2\|\| \langle x\rangle^{-2}e^{\pm i t|p|}\partial_iF_1(|p|\leq t^\alpha)\langle x\rangle^{-2}\| \| \langle x\rangle h^{ij}(x,t)\|\|\partial_j u(t)\|\nonumber\\
 \lesssim_n & \frac{1}{t^{2-3\alpha}} \| \langle x\rangle^2 h^{ij}(x,s)\|_{L^\infty_{x,s}(\mathbb R^{n+1})}\sup\limits_{s\geq 0}\| u(s)\|_{\dot H^1_x(\mathbb R^n)}\nonumber\\
 \in & L^1_t[1,\infty)\label{eq: pert}
 \end{align}
 for all $n\geq 4$, $\alpha\in (0,1/3)$, where we utilized the assumed bound~\eqref{eq: thm: 3, eq 1}. By following verbatim the arguments of Theorem~\ref{main theorem 3} and using~\eqref{est: decay non-local: goal} of Lemma~\ref{lem: decay} for the interaction terms of the present Theorem and~\eqref{eq: pert}~(instead of just using~\eqref{est: decay non-local: goal} and~\eqref{est: decay non-local: goalD}) we obtain the existence of $\Omega_{\alpha}^* \vec{u}(0)$ in $\mathcal H$. 
 
\bigskip
\paragraph{\bf Acknowledgment}

The second author is partially supported by NSF-DMS-220931. The third author is partially supported by ARC-FL220100072, NSF-DMS-220931 and NSERC Grant NA7901.  The authors thank Maxime Van de Moortel for careful reading and for useful discussions.

Parts of this work were done while the third author was at the Fields Institute  for Research in Mathematical Sciences, Toronto, Texas A\&M University, Rutgers University and University of Toronto.

\subsection{Data Availability Statement.}
Data sharing is not applicable to this article as no data sets
were generated or analysed during the current study.

\subsection{Conflict of Interest.} No conflict of interest exits in the submission of this manuscript. The
work was original research that has not been published previously, and not under consideration
for publication elsewhere, in whole or in part. All the authors listed have approved the manuscript.

\bibliographystyle{abbrv}

\end{document}